\newif\ifsoda
\sodafalse

\title{The Unbearable Hardness of Unknotting\thanks{
AdM is partially
supported by the French ANR project ANR-16-CE40-0009-01 (GATO) and the CNRS PEPS project COMP3D. AdM and MT are
partially supported by the Czech-French collaboration project EMBEDS II (CZ:
7AMB17FR029, FR: 38087RM). This work was partially supported by a grant from
the Simons Foundation (grant number 283495 to Yo'av Rieck).
MT is partially supported 
by the project CE-ITI (GA\v{C}R P202/12/G061) and
by the Charles University projects PRIMUS/17/SCI/3 and UNCE/SCI/004.}}

\ifsoda
\documentclass[twoside,leqno,twocolumn]{article}
\usepackage{ltexpprt}

\makeatletter
\def\proof{\@ifnextchar[
  {\@xproof}{\@proof}}

\def\@proof{\trivlist \item[\hskip\labelsep {\em Proof.}  ]
  \ignorespaces }
\def\@xproof[#1]{\trivlist \item[\hskip\labelsep{\em #1.}]
  \ignorespaces }
\makeatother
\else
\documentclass[11pt,a4paper]{article}
\usepackage{a4wide}
\fi

\usepackage{latexsym}
\usepackage{amssymb,amsmath}
\usepackage{graphicx}
\usepackage{url}
\usepackage{color}
\usepackage{verbatim}
\usepackage{epsfig}
\usepackage{enumerate}
\usepackage{hyperref}
\usepackage{environ}
\usepackage{multirow} 

\ifsoda

\author{Arnaud de Mesmay\thanks{Univ. Grenoble Alpes, CNRS, Grenoble INP, GIPSA-lab, 38000 Grenoble, France.}
\and
Yo'av Rieck\thanks{Department of Mathematical Sciences, University of Arkansas Fayetteville, AR 72701, USA.}
\and
Eric Sedgwick\thanks{School of Computing, DePaul University, 243 S.~Wabash Ave, Chicago, IL 60604, USA.}
\and
Martin Tancer\thanks{Department of Applied Mathematics, Charles University, Malostransk\'{e} n\'{a}m. 25, 118~00~~Praha~1, Czech Republic.}}
\else
\usepackage{authblk}
\author[1]{Arnaud de Mesmay}
\author[2]{Yo'av Rieck}
\author[3]{Eric Sedgwick}
\author[4]{Martin Tancer}
\affil[1]{\small{Univ. Grenoble Alpes, CNRS, Grenoble INP\footnote{Institute of Engineering Univ. Grenoble Alpes}, GIPSA-lab, 38000 Grenoble, France}}
\affil[2]{\small{Department of Mathematical Sciences, University of Arkansas
Fayetteville, AR 72701, USA}}
\affil[3]{\small{School of Computing, DePaul University, 243 S.~Wabash Ave,
Chicago, IL 60604, USA}}
\affil[4]{\small Department
of Applied Mathematics,
Charles University, Malostransk\'{e} n\'{a}m.
25, 118~00~~Praha~1, Czech Republic.}
\fi

\ifsoda
\date{}
\else
\date{\today}
\fi

\newcommand{\NP}{\textbf{NP}}

\newcommand{\DD}{\mathbf D}
\newcommand{\rr}{\mathbf r}

\newcommand{\R}{{\mathbb{R}}}

\newcommand{\Z}{{\mathbb{Z}}}

\newcommand{\Imin}{\ensuremath{\rm{I}^-}}
\newcommand{\Ipl}{\ensuremath{\rm{I}^+}}
\newcommand{\Imove}{\ensuremath{\rm{I}}}
\newcommand{\II}{\ensuremath{\rm{II}}}
\newcommand{\III}{\ensuremath{\rm{III}}}
\newcommand{\IImin}{\ensuremath{\rm{II}^-}}
\newcommand{\IIpl}{\ensuremath{\rm{II}^+}}

\DeclareMathOperator{\TRUE}{TRUE}
\DeclareMathOperator{\FALSE}{FALSE}

\DeclareMathOperator{\reid}{rm} 
\DeclareMathOperator{\defe}{def} 
\DeclareMathOperator{\cross}{cross}

\ifsoda

\newtheorem{theorem_soda}{Theorem}[section]

\newtheorem{cor}[theorem_soda]{Corollary}
\newtheorem{notation}[theorem_soda]{Notation}
\newtheorem{example}[theorem_soda]{Example}
\newtheorem{claim}[theorem_soda]{Claim}
\newtheorem{remark}[theorem_soda]{Remark}
\newtheorem{definition}[theorem_soda]{Definition}
\else
\usepackage{amsthm}
\usepackage{thmtools}
\usepackage{thm-restate}
\theoremstyle{plain}

\newtheorem{theorem}{Theorem}

\newtheorem{lemma}[theorem]{Lemma}

\newtheorem{claim}{Claim}[theorem]

\newtheorem{obs}[theorem]{Observation}
\theoremstyle{definition}

\newtheorem{definition}[theorem]{Definition}

\theoremstyle{definition}
\newtheorem{remark}[theorem]{Remark}

\fi

\let\myinput\input

\newcommand{\includesvg}[1]{%
\myinput{#1.pdf_tex}%
}

\long\def\onefigure#1#2{
\begin{figure*}[tbp]
\begin{center}
#1
\end{center}
\caption{#2}
\end{figure*}
}
\graphicspath{{./}{./LinksHard/}}
\def\immediateFigure#1{%
\smallskip\begin{center}#1\end{center}\smallskip }

\newcommand{\labfig}[2]  
{\onefigure{\mbox{\includegraphics{#1}}}{\label{f:#1} #2} }

\newcommand{\labfigw}[3]  
{\onefigure{\mbox{\includegraphics[width=#2]{#1}}}{\label{f:#1} #3}}

\newcommand{\immfig}[1]  
{\immediateFigure{\mbox{\includegraphics{#1}}}}

\newcommand{\immfigw}[2] 
{\immediateFigure{\mbox{\includegraphics[width=#2]{#1}}}}

\definecolor{orange}{rgb}{1,0.5,0}

\newif\ifcmts
\cmtstrue     

\ifcmts
\newcommand{\marrow}{\marginpar{\boldmath$\longleftarrow$}}
\newcommand{\eric}[1]{\ifhmode\newline\fi\marrow \textsf{\textcolor{red}{\bf ERIC:} #1\newline}}
\newcommand{\martin}[1]{\ifhmode\newline\fi\marrow \textsf{\textcolor{magenta}{\bf
MARTIN:} #1\newline}}
\newcommand{\arnaud}[1]{\ifhmode\newline\fi\marrow \textsf{\textcolor{blue}{\bf
ARNAUD:} #1\newline}}
\newcommand{\yoav}[1]{\ifhmode\newline\fi\marrow \textsf{\textcolor{orange}{\bf
YO'AV:} #1\newline}}
\newcommand{\red}[1]{{\textcolor{red}{#1}}}
\newcommand{\blue}[1]{{\textcolor{blue}{#1}}}

\newcommand{\orange}[1]{{\textcolor{orange}{#1}}}
\else
\newcommand{\marrow}{}
\newcommand{\yoav}[1]{}
\newcommand{\eric}[1]{}
\newcommand{\arnaud}[1]{}
\newcommand{\martin}[1]{}
\newcommand{\red}[1]{#1}
\newcommand{\blue}[1]{#1}

\newcommand{\orange}[1]{#1}
\fi

\begin{document}

\newif\ifSetUp
\SetUpfalse

\maketitle
\ifsoda
\fancyfoot[R]{\footnotesize{\textbf{Copyright for this paper retained by the
authors.}}}
\else
\thispagestyle{empty}
\fi
\begin{abstract}

We prove that deciding if a diagram of the unknot can be untangled using at most \(k\) Riedemeister moves (where $k$ is part of the input) is \NP-hard.
We also prove that several natural questions regarding links in the
\(3\)-sphere are \NP-hard, including detecting whether a link contains a
trivial sublink with $n$ components, computing the unlinking number of a link,
and computing a variety of link invariants related to four-dimensional topology
(such as the \(4\)-ball Euler characteristic, the slicing number, and the \(4\)-dimensional clasp number).

\end{abstract}

\tableofcontents

\section{Introduction}

\paragraph{Unknot recognition via Reidemeister moves.}
The \em unknot recognition problem \em asks whether a given knot is the unknot.
Decidability of the unknot recognition problem was established by
Haken~\cite{h-tn-61}, and since then several other algorithms were
constructed (see for example the survey of Lackenby~\cite{l-ekt-17}).

One can ask, naively, if one can decide if a given knot diagram represents the
unknot simply by untangling the diagram: trying various Reidemeister moves
until there are no more crossings. A first issue is that one might need to \emph{increase} the number of crossings at some point in this untangling: examples of ``hard unknots'' witnessing this necessity can be found in Kaufman and Lambropoulou~\cite{kl-huct-14} . The problem then, obviously, is knowing when to
stop: if we have not been able to untangle the diagram using so many moves, is
the knot in question necessarily knotted or should we keep on trying? 

In~\cite{hl-nrmnu-01}, Hass and Lagarias gave an explicit (albeit rather large)
bound on the number of Reidemeister moves needed to untangle a diagram of the
unknot. Lackenby~\cite{l-pubrm-15} improved the bound to polynomial thus
showing that the unknot recognition problem is in \NP\ (this was previously proved
in~\cite{hlp-ccklp-99}).  The unknot recognition problem is also in
co-\NP~\cite{l-ecktn-16} (assuming the Generalized Riemann Hypothesis,
this was previously shown by Kuperberg~\cite{Kuperberg}). Thus if the
unknot recognition problem were  \NP-complete (or co-\NP-complete) we would
have that \NP\ and co-\NP\ coincide which is commonly believed not to be the
case.  This suggests that the unknot recognition problem is \em not \em
\NP-hard.

It is therefore natural to ask if there is a way to use Reidemeister moves
leading to a better solution than a generic brute-force search.
Our main result suggests that there may be
serious difficulties in such an approach: given a 3-SAT instance \(\Phi\) we
construct an unknot diagram and a number \(k\), so that the diagram can be
untangled using at most \(k\) Reidemeister moves if and only if \(\Phi\) is
satisfiable.  Hence any algorithm that can calculate the minimal number of
Reidemeister moves needed to untangle unknot diagrams will be robust enough to
tackle any problem in \NP.

The main result of this paper is:

\begin{theorem}
\label{t:main}
  Given an unknot diagram $D$ and an integer $k$,
  deciding if $D$ can be untangled using at most $k$ Reidemeister moves is
  \NP-complete.
\end{theorem}

Lackenby~\cite{l-pubrm-15} proved that the problem above is in \NP, and therefore we only need to show \NP-hardness.

For the reduction in the proof of Theorem~\ref{t:main} we have to construct arbitrarily large diagrams of the unknot.  The difficulty in the proof is to establish tools powerful enough to provide useful lower bounds on the minimal number of Reidemeister moves needed to untangle these diagrams. For instance, the algebraic methods of Hass and Nowik~\cite{hn-udrqnr-10} are not strong enough for our reduction.  It is also quite easy to modify the construction and give more easily lower bounds on the number of Reidemeister moves needed to untangle unlinks if one allows the use of arbitrarily many components of diagrams with constant size, but those techniques too cannot be used for Theorem~\ref{t:main}.  We develop the necessary tools in Section~\ref{ss:defect}.

\paragraph{Computational problems for links.} Our approach for proving
Theorem~\ref{t:main} partially builds on techniques to encode satisfiability instances using Hopf links and Borromean rings, that we previously used in~\cite{dmrst-eR3NPh-18} (though the technical details are very different). With these techniques, we also show that a variety of link invariants are \NP-hard to compute.

Precisely, we prove:

\begin{theorem}\label{t:main2}
Given a link diagram $L$ and an integer $k$, the following problems are \NP-hard:
\begin{enumerate}[(a)]
\item deciding whether $L$ admits a trivial unlink with $k$ components as a sublink.
\item deciding whether an intermediate invariant has value $k$ on $L$,
\item deciding whether $\chi_4(L)=0$,
\item deciding whether $L$ admits a smoothly slice sublink with $k$ components.
\end{enumerate}
\end{theorem}

We refer to Definition~\ref{dfn:IntermediateInvariant} for the definitions of  $\chi_4(L)$, the $4$-ball Euler characteristic,  and of intermediate invariants. These are broadly related to the topology of the $4$-ball, and include the unlinking number, the ribbon number, the slicing number, the concordance unlinking number, the concordance ribbon number, the concordance slicing number, and the  \(4\)-dimensional clasp number. See, for example,~\cite{Shibuya} for a discussion of many intermediate invariants.

\paragraph{Related complexity results.} The complexity of computational problems pertaining to knots and links is quite poorly understood. In particular, only very few computational lower bounds are known, and as far as we know, almost none concern classical knots (i.e., knots embedded in ${S}^3$): apart from our Theorem~\ref{t:main}, the only other such hardness proof we know of~\cite{ks-cikloi-18,s-cce3mi-18} concerns counting coloring invariants (i.e., representations of the fundamental group) of knots. More lower bounds are known for classical links. Lackenby~\cite{l-schpl3m-17} showed that determining if a link is a sublink of another one is \NP-hard. Our results strengthen this by showing that even finding an $n$-component unlink as a sublink is already \NP-hard. Agol, Hass and Thurston~\cite{aht-cckgs-06} showed that computing the genus of a knot in a $3$-manifold is \NP-hard, and Lackenby~\cite{l-schpl3m-17} showed that computing the Thurston complexity of a link in ${S}^3$ is also \NP-hard. Our results complement this by showing that the $4$-dimensional version of this problem is also \NP-hard. 

Regarding upper bounds, the current state of knowledge is only slightly better. While, as we mentioned before, it is now known that the unknot recognition problem is in $\NP \cap$ co-\NP, many natural link invariants are not even known to be decidable. In particular, this is the case for all the invariants for which we prove $\NP$-hardness, except for the problem of finding the maximal number of components of a link that form an unlink, which is in $\NP$ (see Theorem~\ref{thm:TirvialSublinkNPcomplete}).

Shortly before we finished our manuscript, Koenig and Tsvietkova
posted a preprint~\cite{kt-NPhnak-18} that also shows that certain computational problems on links are \NP-hard, with
some overlap with the results obtained in this paper (the trivial sublink
problem and the unlinking number). They also show \NP-hardness of computing the number of
Reidemeister moves for getting between two diagrams of the unlink, but their construction does not untangle the diagram and requires arbitrarily
many components. 
Theorem~\ref{t:main} of the current paper is stronger and answers 
Question~17 of~\cite{kt-NPhnak-18}.

\paragraph*{Organization.} This paper is organized as follows. After some
preliminaries in Section~\ref{s:prelim}, we start by proving the hardness of
the trivial sublink problem in Part~\ref{p:sublink} because it is very simple
and provides a good introduction for our other reductions. We then proceed to
prove Theorem~\ref{t:main} in Part~\ref{p:knots} and the hardness of the
unlinking number and the other invariants in Part~\ref{p:links}. 
The three parts are independent and the reader can read any one part
alone.

\section{Preliminaries}\label{s:prelim}

\paragraph{Notation.} Most of the notation we use is standard.  By \em knot \em we mean a tame piecewise linear embedding of the circle \(S^{1}\) into the \(3\)-sphere \(S^{3}\).  By \em link \em we mean a tame, piecewise linear embedding of the disjoint union of any finite number of copies of \(S^{1}\). 
We use interval notation for natural numbers, for example, \([3,5]\) means \(\{3,4,5\}\); we use \([12]\) to indicate \([1,12]\).
We assume basic familiarity with computational complexity and knot theory, and refer to basic textbooks such as Arora and Barak~\cite{ab-ccma-09} for the former and Rolfsen~\cite{rolfsen} for the latter.

\paragraph{Diagram of a knot or a link.}
All the computational problems that we study in this paper take as input the \emph{diagram} of a knot or a link, which we define here.

A diagram of a knot is a piecewise linear map $D\colon S^1 \to \R^2$ in general position; for such a map, every point in $\R^2$ has at most two preimages, and there are finitely many points in $\R^2$ with exactly two preimages (called \em crossing\em).  Locally at crossing 
two arcs cross each other transversely, and the diagram contains the information of which  
arc passes `over' and which `under'. This we usually depict by interrupting the arc that passes under. (A diagram usually arises as a composition of a (piecewise linear) knot $\kappa \colon S^1 \to \R^3$ and a generic projection $\pi \colon \R^3 \to \R^2$ which also induces `over' and `under' information.) We usually identify a diagram $D$ with its image in $\R^2$ together with the information about underpasses/overpasses at crossings; see, for example, Figure~\ref{f:vg}, ignoring the notation on the picture. Diagrams are considered up-to isotopy.

Similarly, a diagram of a link is a piecewise linear map $D \colon \coprod S^1 \to \R^2$ in general position, where $\coprod$ denotes a disjoint union of a finite number of circles $S^1$, and with the same additional information at the crossings.

By an \emph{arc} in the diagram $D$ we mean a set $D(\alpha)$ where $\alpha$
is an arc in $S^1$ (i.e., a subset of $S^1$ homeomorphic to the closed interval).

The size of a knot or a link diagram is its number of crossings plus number of components
of the link. Up to a constant factor, this complexity exactly describes the complexity of encoding the combinatorial information contained in a knot or link diagram.

\paragraph*{$3$-satisfiability.} A formula in \emph{conjunctive normal form} in variables $x_1, \dots, x_n$ is a
Boolean formula of the form $c_1 \wedge c_2 \wedge \cdots \wedge c_m$ where each
$c_i$ is a \emph{clause}, that is, a formula of the form $(\ell_1 \vee \ell_2
\vee \cdots \vee \ell_k)$ where each $\ell_j$ is a \emph{literal}, that is, a
variable $x_t$ or its negation $\neg x_t$. A formula $\Phi$ is \emph{satisfiable} if
there is an assignment to the variables (each variable is assigned $\TRUE$ or
$\FALSE$) such that $\Phi$ evaluates to $\TRUE$ in the given assignment.

A 3-SAT problem is the well-known \NP-hard problem. On input there is a formula
$\Phi$ in conjunctive normal form such that every clause contains
exactly\footnote{Here we adopt a convention from~\cite{papadimitriou94}.
Some other authors define require only `at most three' literals.} three
variables; see, e.g.,~\cite[Proposition 9.2]{papadimitriou94}. 

\part{Trivial sublink}
\label{p:sublink}

Informally, the trivial sublink problem asks, given a link \(L\) and a positive integer \(n\), whether \(L\) admits the \(n\)-component unlink as a sublink.  We define:

\begin{definition}[The Trivial Sublink Problem]
\label{dfn:TrivialSublinkProblem}
An \em unlink, \em or a \em trivial link, \em is a link in \(S^{3}\) whose
components bound disjointly embedded disks.   A \em trivial sublink \em of a
link \(L\) is an unlink formed by a subset of the components of \(L\).  The \em
trivial sublink problem \em  asks, given a link \(L\) and a positive integer
\(n\), whether \(L\) admits an \(n\) component trivial sublink.

\end{definition}

\begin{theorem}
\label{thm:TirvialSublinkNPcomplete}
The trivial sublink problem is \NP-complete. 
\end{theorem}

Note that Theorem~\ref{thm:TirvialSublinkNPcomplete} is just a slight
extension of Theorem~\ref{t:main2}(a), claiming also \NP-membership. The
essential part is \NP-hardness.

\begin{proof}
It follows from Hass, Lagarias, and Pippenger~\cite{hlp-ccklp-99} that deciding
if a link is trivial is in \NP.   By adding to their certificate a collection
of \(n\) components of \(L\) we obtain a certificate for the trivial sublink
problem, showing that it is in \NP\
(\NP-membership of the trivial sublink problem was also established, using completely different techniques, by Lachenby~\cite{l-pubrm-15}).
Thus all we need to show is that the
problem is \NP-hard.  We will show this by reducing 3-SAT to the trivial
sublink problem. 

Given a 3-SAT instance \(\Phi\), with \(n\) variables (say \(x_{1},\dots,x_{n}\)) and \(m\) clauses, we construct a diagram \(D_{\Phi}\) as follows (see Figure~\ref{figure:Lphi}):
\begin{figure}
\begin{center}
\def\svgwidth{\textwidth}
\includesvg{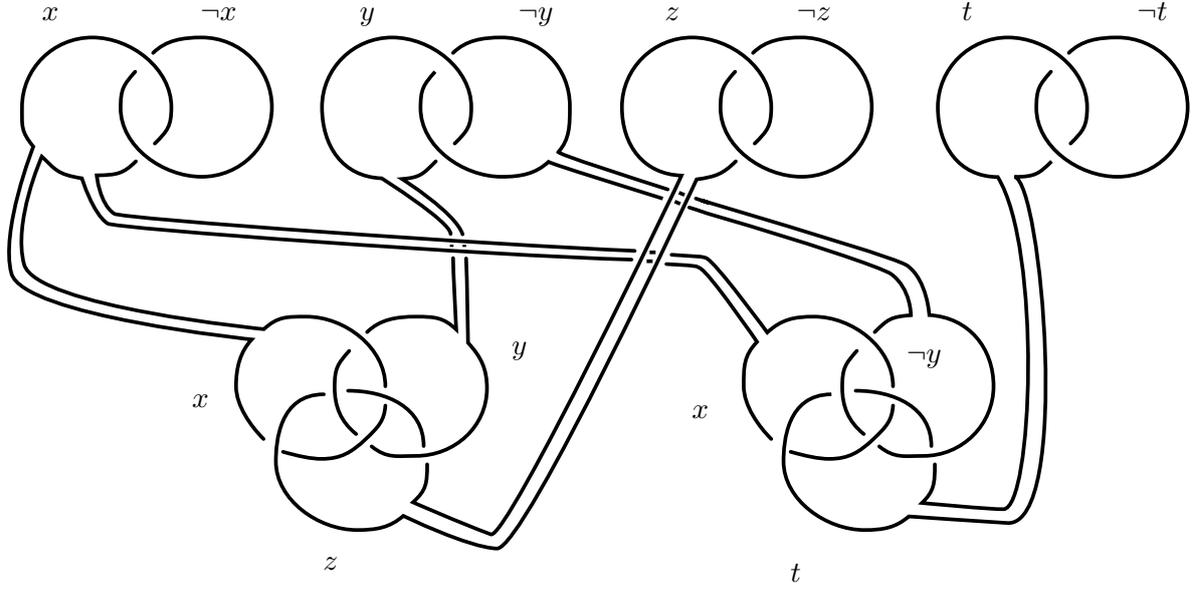}
  \caption{\(D_{\Phi}\) for \(\Phi = (x \vee y \vee z) \wedge (x \vee \neg y \vee t)\))}
\label{figure:Lphi}
\end{center}
\end{figure} 
we first mark \(n+m\) disjoint disks in the plane.  In each of the first \(n\)
disks we draw a diagram of the Hopf link, marking the components in the \(i\)th
disk as \(\kappa_{x_{i}}\) and  \(\kappa_{\neg x_{i}}\).  In the remaining
\(m\) disks we draw diagrams of the Borromean rings and label them according to
the clauses of \(\Phi\).  We now band each component of the Borromean rings to the Hopf link component with the same label.  Whenever two bands cross we have one move over the other (with no ``weaving''); we assume, as we may, that no two bands cross twice.  It is easy to see that this can be done in polynomial time.  The diagram we obtain is \(D_{\Phi}\) and the link it represents is denoted \(L_{\Phi}\) (note that \(L_{\Phi}\) has exactly \(2n\) components).  We complete the proof by showing that \(L_{\Phi}\) admits an \(n\)-component trivial sublink exactly when \(\Phi\) is satisfiable.

\begin{claim}
\label{claim:SatImpliesTrivialUnlink}
If \(\Phi\) is satisfiable then \(L_{\Phi}\) admits an \(n\)-component trivial sublink.
\end{claim}

Given a satisfying assignment we remove from \(L_{\Phi}\) the components that correspond to satisfied literals, that is, 
if \(x_{i} = \TRUE\) we remove \(\kappa_{x_{i}}\) from \(L_{\Phi}\)
and if \(x_{i} = \FALSE\) we remove \(\kappa_{\neg x_{i}}\) from \(L_{\Phi}\).
We claim that the remaining \(n\) components form an unlink.  To see this,
first note that since the assignment is satisfying, from each copy of the Borromean rings at least one component was removed.  Therefore the rings fall apart and (since we did not allow ``weaving'') the diagram obtained retracts into the first \(n\) disks.  In each of these disks we had, originally, a copy of the Hopf link; by construction exactly one component was removed.  This shows that the link obtained is indeed the \(n\)-component unlink;  Claim~\ref{claim:SatImpliesTrivialUnlink} follows.

\begin{claim}
\label{claim:TrivialUnlinkImpliesSat}
If  \(L_{\Phi}\) admits an \(n\)-component trivial sublink, then \(\Phi\) is satisfiable.
\end{claim}

Suppose that \(L_{\Phi}\) admits an \(n\)-component trivial sublink
\(\mathcal{U}\).   Since \(\mathcal{U}\) itself does not admit the Hopf link as
a sublink, for each \(i\), at most one of \(\kappa_{x_{i}}\) and \(\kappa_{\neg
x_{i}}\) is in \(\mathcal{U}\).  Since \(\mathcal{U}\) has \(n\) components we
see that \em exactly \em one of \(\kappa_{x_{i}}\) and \(\kappa_{\neg x_{i}}\)
is in \(\mathcal{U}\).  If  \(\kappa_{x_{i}}\) is in \(\mathcal{U}\) we set
\(x_{i} = \FALSE\) and if \(\kappa_{\neg x_{i}}\) is in \(\mathcal{U}\) we set
\(x_{i} = \TRUE\).  Now since \(\mathcal{U}\) does not admit the Borromean
rings as a sublink, from each copy of the Borromean rings at least one component is not in \(\mathcal{U}\).  It follows that in each clause of \(\Phi\) at least one literal is satisfied, that is, the assignment satisfies \(\Phi\);
Claim~\ref{claim:TrivialUnlinkImpliesSat} follows.

This completes the proof of Theorem~\ref{thm:TirvialSublinkNPcomplete}.
\end{proof}


\part{The number of Reidemeister moves for untangling}
\label{p:knots}

\section{A restricted form of the satisfiability problem} 
For the proof of Theorem~\ref{t:main}, we will need a slightly restricted form
of the 3-SAT problem given by lemma below.

\begin{lemma}
\label{l:sat_variant}
Deciding whether a formula $\Phi$ in conjunctive normal form is satisfiable
  is \NP-hard even if we assume the following conditions on $\Phi$. 
\begin{itemize}
 \item Each clause contains exactly three literals.
 \item No clause contains both $x$ and $\neg x$ for some variable $x$.
 \item Each pair of literals $\{\ell_1,\ell_2\}$ occurs in at most one
   clause.
\end{itemize}
\end{lemma}

\begin{proof}
The first condition says that we consider the 3-SAT problem. Any clause
violating the second condition can be removed from the formula without
affecting satisfiability of the formula as such clause is always satisfied. 
Therefore, it is
  sufficient to provide a recipe to build in polynomial time a formula $\Phi$ satisfying the
  three conditions above out of a formula $\Phi'$ satisfying only the first two
  conditions.

  First, we consider an auxiliary formula 
  $$\Psi = \Psi(t,a,b,c) = (t \vee a \vee \neg b) \wedge
  (t \vee b \vee \neg c) \wedge (t \vee c \vee \neg a) \wedge (a \vee b \vee c)
  \wedge(\neg a \vee \neg b \vee \neg c).$$
We observe that for any satisfying assignment of $\Psi$ we get that $t$ is
  assigned $\TRUE$. Indeed, if $t$ were assigned $\FALSE$ then $(t \vee a \vee
  \neg b) \wedge
    (t \vee b \vee \neg c) \wedge (t \vee c \vee \neg a)$ translates as $(b
    \Rightarrow a) \wedge (c \Rightarrow b) \wedge (a \Rightarrow c)$, that is
    all $a$, $b$, and $c$ are equivalent. However, then $(a \vee b \vee c)
      \wedge(\neg a \vee \neg b \vee \neg c)$ cannot be satisfied.

On the other hand, we also observe that there is a satisfying assignment for
  $\Psi$ where $t$ is assigned $\TRUE$ and, for example, it is sufficient to
  assign $a$ with $\TRUE$, $b$ with $\FALSE$ and $c$ arbitrarily.

Now we return to the formula $\Phi'$ discussed above. 
Suppose there exists a pair  literals $\ell_1$ and $\ell_2$ contained in two clauses of $\Phi'$, say $(\ell_1 \vee \ell_2 \vee \ell_3)$ and $(\ell_1 \vee
  \ell_2 \vee \ell_4)$.  
We replace them with
$$
  (\ell_1 \vee \ell_2 \vee x) \wedge (\ell_3 \vee \neg x \vee \neg t_1) \wedge
  (\ell_4 \vee \neg x \vee \neg t_2) \wedge \Psi(t_1, a_1,b_1,c_1) \wedge \Psi(t_2,
  a_2, b_2, c_2),
$$
where $x, t_1, t_2, a_1, a_2, b_1, b_2, c_1,$ and $c_2$ are newly added
  variables, obtaining a new formula $\Phi''$. We aim to show that $\Phi''$ is
  satisfiable if and only if $\Phi'$ is satisfiable.

Let us first assume that $\Phi'$ is satisfiable and fix a satisfying
  assignment. If $\ell_1 \vee \ell_2$ is $\TRUE$ in this assignment, then we
  may extend it to variables of $\Phi''$ by setting $x$ to $\FALSE$, $t_1$ and
  $t_2$ to $\TRUE$ 
  and $a_1, \dots, c_2$ so that $\Psi(t_1, a_1,b_1,c_1) \wedge  \Psi(t_2, a_2, b_2, c_2)$ is satisfied. 
If $\ell_1 \vee \ell_2$ is $\FALSE$
  in the assignment for $\Phi'$, then $\ell_3$ and $\ell_4$ must be assigned
  $\TRUE$. Then we may extend by setting $x$ to $\TRUE$ and 
 $t_1, t_2, a_1, \dots, c_2$ as before.

  Now let us assume that $\Phi''$ is satisfiable. Then $t_1$ and $t_2$ are set
  to $\TRUE$ due to the properties of $\Psi$. If $x$ is $\TRUE$, then both
  $\ell_3$ and $\ell_4$ must be true and the restriction of the assignment on
  $\Phi''$ to the variables of $\Phi'$ is therefore a satisfying assignment
  for $\Phi'$. Similarly, if $x$ is $\FALSE$, then $\ell_1 \vee \ell_2$ must be
  true and the restriction is again a satisfying assignment for $\Phi'$.

We also observe that in $\Phi''$ we have reduced the number of clauses
  containing the literals $\ell_1$ and $\ell_2$ simultaneously, and for any
  other pair of literals we do not increase the number of clauses containing
  that pair. Therefore, after a polynomial number of steps, when always adding
  new variables, we arrive at a desired formula $\Phi$ satisfying all three conditions.
\end{proof}

\section{The defect}
\label{ss:defect}

\paragraph{Reidemeister moves.}
Reidemeister moves are local modifications of a diagram depicted in
Figure~\ref{f:rm} (the labels at the crossings in a $\III$ move will be
used only later on). We distinguish the $\Imove$ move (left), the $\II$ move (middle) and
the $\III$ move (right). The first two moves affect the number of crossings, thus we
further distinguish the $\Imin$ and the $\IImin$ moves which reduce the number of crossings
from the $\Ipl$ and the $\IIpl$ moves which increase the number of crossings.

\begin{figure}
\begin{center}
  \includegraphics{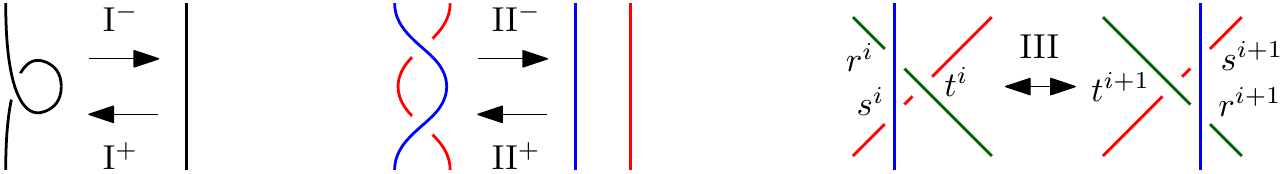}
  \caption{Reidemeister moves} 
\label{f:rm}
\end{center}
\end{figure}

\paragraph{The number of Reidemeister moves for untangling a knot.}
A diagram of an unknot is \emph{untangled} if it does not contain any crossings.
The untangled diagram is denoted by $U$. Given a diagram $D$ of an unknot, 
an \emph{untangling} of $D$ is a sequence $\DD = (D^0, \dots, D^k)$ where $D^0 =
D$, $D^k = U$ (recall that diagrams are only considered up to isotopy) and $D^i$ is obtained from $D^{i-1}$ by a single
Reidemeister move. The number of Reidemeister moves in $\DD$ is denoted by
$\reid(\DD)$, that is, $\reid(\DD) = k$. We also define $\reid(D) := \min
\reid(\DD)$ where the minimum is taken over all untanglings $\DD$ of $D$. 

\paragraph{The defect.} Let us denote by $\cross(D)$ the number of crossings in
$D$. Then the \emph{defect} of an untangling 
$\DD$ is defined by the formula
$$
\defe(\DD) := 2\reid(\DD) - \cross(D).
$$
The \emph{defect} of a diagram $D$ is defined as $\defe(D) := 2\reid(D) -
\cross(D)$. Equivalently, $\defe(D) = \min
\defe(\DD)$ where the minimum is taken over all untanglings $\DD$ of $D$. 

The defect is a convenient way to reparametrize the number of Reidemeister
moves due to the following observation.

\begin{obs}
For any diagram $D$ of the unknot and any untangling $\DD$ of $D$ we have
  $\defe(\DD) \geq 0$. Equality holds if
and only if $\DD$ uses   only $\IImin$ moves.
\end{obs}

\begin{proof} Every Reidemeister move in $\DD = (D^0, \dots, D^k)$ removes at most two crossings and the $\IImin$ move is the only move that removes exactly two crossings.
  Therefore, the number of crossings in $D = D^0$ is at most $2k$ and 
  equality holds if and only if every move is a  $\IImin$ move.
\end{proof}

\paragraph{Crossings contributing to the defect.}
Let $\DD = (D^0, \dots, D^k)$ be an untangling of a diagram $D = D^0$ of an unknot. 

Given a crossing $r^i$ in $D^i$, for $0 \leq i \leq k-1$, it may vanish by the
move transforming $D^i$ into $D^{i+1}$ if this is a $\Imin$
or a $\IImin$ move affecting the crossing. In all other cases it
\emph{survives} and we denote by $r^{i+1}$ the corresponding crossing in
$D_{i+1}$.  Note that in the case of a $\III$ move there are three crossings affected by the move and three crossings afterwards. Both before and after,
each crossing is the unique intersection between a pair of the three arcs of
the knot that appear in this portion of the diagram. So we may say that these
three crossings survive the move though they change their actual geometric
positions (they swap the order in which
they occur along each of the three arcs); see Figure~\ref{f:rm}.

With a slight abuse of terminology, by a crossing in $\DD$ we mean a maximal
sequence $\rr
= (r^a, r^{a+1}, \dots, r^b)$ such that $r^{i+1}$ is the crossing in $D^{i+1}$ 
corresponding to $r^i$ in $D^i$ for any $i \in [a,b-1]$. By maximality we mean,
that $r^b$ vanishes after the $(b+1)$st move and either $a = 0$ or $r^a$ is
introduced by the $a$th
Reidemeister move (which must be a $\Ipl$ or $\IIpl$ move). 

An \emph{initial crossing} is a crossing $\rr = (r^0, r^{1}, \dots, r^b)$ in
$\DD$. Initial crossings in $\DD$ are in one-to-one correspondence with
crossings in $D = D^0$. For simplicity of notation, $r^0$ is also denoted $r$
(as a crossing in $D$).

A Reidemeister $\IImin$ move in $\DD$ is \emph{economical}, if both crossings removed by
this move are initial crossings; otherwise, it is \emph{wasteful}.

Let $m_3(\rr)$ be the number of $\III$ moves affecting a crossing $\rr$. 
The \emph{weight} of an initial crossing $\rr$ is defined
in the following way.
$$
w(\rr) = \frac23 m_3(\rr) + 
\left\{
  	\begin{array}{ll}
	  		0  & \mbox{if $\rr$ vanishes by an economical $\IImin$
	move};  \\
	  		1  & \mbox{if $\rr$ vanishes by a $\Imin$ move};  \\
	  		2  & \mbox{if $\rr$ vanishes by a wasteful $\IImin$
      move}.  \\
	  	\end{array}
		\right.
$$

For later purposes, we also define $w(r) := w(\rr)$ and $w(R) := \sum_{r \in R}
w(r)$ for a subset $R$ of the set of all crossings in $D$. 

\begin{lemma}
\label{l:discharging}
  Let $\DD$ be an untangling of a diagram $D$. Then 
  $$
  \defe(\DD) \geq \sum\limits_{\rr} w(\rr),
  $$
  where the sum is over all initial crossings $\rr$ of $\DD$.
\end{lemma}

\begin{proof}
In the proof we use the discharging technique, common in
graph theory.

Let us put charges on crossings in $\DD$ and on Reidemeister moves used in $\DD$.
The initial charge will be 
\begin{itemize}
\item[$2$] on each Reidemeister move;
\item[$-1$] on each initial crossing; and
\item[$0$] on each non-initial crossing.
\end{itemize}

We remark that the sum of the initial charges equals to $\defe(\DD)$ by the
definition of the defect.

Now we start redistributing the charge according to the rules described below. The aim is that, after the redistribution, the charge on each initial crossing will
be at least $w(\rr)$ and it will be at least $0$ on non-initial crossings and on
Reidemeister moves. This will prove the lemma, as the sum of the charges after 
the redistribution is still equal to the defect, whereas it will be at least
the sum of the weights of initial crossings.

We apply the following rules for the redistribution of the charge.

\begin{itemize}
  \item[(R1)] Every $\Ipl$ move sends charge $2$ to the (non-initial) crossing it
  creates.
\item[(R2)] Every $\Imin$ move sends charge $2$ to the crossing it removes.
\item[(R3)] Every $\IIpl$ move sends charge $1$ to each of the two (non-initial)
  crossings it creates.
\item[(R4)] Every economical $\IImin$ move sends charge $1$ to each of the two
  (initial) crossings it removes.
\item[(R5)] Every wasteful $\IImin$ move that removes exactly one initial crossing
  sends charge $3$ to this initial crossing.
\item[(R6)] Every $\III$ move sends $\frac 23$ to every crossing it affects.
\item[(R7)] Every non-initial crossing which is removed by a wasteful $\IImin$ move
  sends charge $1$ to this move.
\end{itemize}

Now it is routine to check that the desired conditions are satisfied, which we now explain. 

Every move has charge at least $0$: The initial charge on moves is $2$. The
rules are set up so that every move distributes charge at most $2$ with
exception of the rule (R5). However, in this case, the wasteful $\IImin$ move
that removes exactly one initial crossing from (R5) gets $1$ charge from rule
(R7).

Every non-initial crossing has charge at least $0$: The initial charge is $0$.
The only rule that depletes the charge is (R7); however in such case, the
charge is replenished by (R1) or (R3).

Every initial crossing $\rr$ has charge equal at least $w(\rr)$: The initial charge
is $-1$. First we observe that (R6) sends the charge $\frac23 m_3(\rr)$ to $\rr$.
If $\rr$ vanishes by an economical $\IImin$ move, it gets additional charge $1$ by
(R4). If $\rr$ vanishes by a $\Imin$ move, it gets additional charge $2$ by (R2).
Finally, if $\rr$ vanishes by a wasteful $\IImin$ move then this move removes
exactly one initial crossing, namely $\rr$. Therefore $\rr$ gets an additional charge
$3$ by (R5).
\end{proof}

We will also need a variant for a previous lemma where we get equality, if we
use the $\Imin$ and $\IImin$ moves only.

\begin{lemma}
\label{l:defect_exact}
  Let $\DD$ be an untangling of a digram $D$ which uses the $\Imin$ and $\IImin$ moves
  only. Then 
  $$
  \defe(\DD) = \sum\limits_{\rr} w(\rr) = \hbox{ number of $\Imin$ moves},
  $$
  where the sum is over all initial crossings $\rr$ of $\DD$.
\end{lemma}

\begin{proof}
 Let $k_1$ be the number of initial crossings removed by a $\Imin$ move and
 $k_2$ the number of initial crossings removed by a $\IImin$ move. Then
 $\reid(\DD) = k_1 + k_2/2$ and from the definition of the defect we get
 $\defe(\DD) = 2(k_1 + k_2/2) - (k_1 + k_2) = k_1$.  
 It follows directly the definition of the weight that 
 $k_{1} = \sum\limits_{\rr} w(\rr)$. 
\end{proof}

\paragraph{Twins and the preimage of a bigon.}
Let $\rr$ be an initial crossing in an untangling $\DD = (D^0, \dots, D^k)$ removed by an economical
$\IImin$ move. The \emph{twin} of $\rr$, denoted by $t(\rr)$ is the other 
crossing in $\DD$ removed by the same $\IImin$ move. Note that $t(\rr)$ is also an
initial crossing (because of the economical move). We also get $t(t(\rr)) =
\rr$. If $\rr = (r^0, r^1, \dots, r^b)$, then we also extend the definition of
a twin to $D^i$ in such a way that $t(r^i)$ is uniquely defined by $t(\rr) =
(t(r^0), \dots, t(r^b))$. In particular, we will often use a twin $t(r)$
of a crossing $r = r^0$ in $D$ (if it exists).

Furthermore, the crossings $r^b$ and $t(r^b)$ in $D^b$ form a
bigon that is removed by the forthcoming $\IImin$ move.  Let $\alpha^b(r)$
and $\beta^b(r)$ be the two arcs of the bigon (with endpoints $r^b$ and
$t(r^b)$) so that $\alpha^b(r)$ is the arc 
that, after extending slightly,
overpasses  the crossings $r^b$ and $t(r^b)$ whereas a slight extension of
$\beta^b(r)$  underpasses these crossings. (The reader may remember this as
$\alpha$ is `above' and $\beta$ is `below'.)\footnote{Note that $\alpha^b(r)$ and
$\beta^b(r)$ are uniquely defined by $r$ as well as by $\rr$. The choice of $r$
in the notation will be more useful later on.} 
Now we can inductively define arcs
$\alpha^i(r)$ and $\beta^i(r)$ for $i \in [0,b-1]$ so that $\alpha^i(r)$ and
$\beta^i(r)$ are the unique arcs between $r^i$ and $t(r^i)$ which are
transformed to (already defined) $\alpha^{i+1}(r)$ and $\beta^{i+1}(r)$ by the
$i$th Reidemeister move. We also set $\alpha(r) = \alpha^0(r)$ and
$\beta(r) =
\beta^0(r)$. Intuitively, $\alpha(r)$ and $\beta(r)$ form a preimage of the
bigon removed by the $(b+1)$st move and they are called the \emph{preimage
arcs} between $r$ and $t(r)$.

\paragraph{Close neighbors.} 
Let $R$ be a subset of the set of crossings in $D$. Let $r$ and $s$ be any two
crossings in $D$ (not necessarily in $R$) and let $c$ be a non-negative
integer. We say that $r$ and $s$ are
  \emph{$c$-close neighbors with respect to $R$} if $r$ and $s$ can be connected by two arcs $\alpha$ and $\beta$ such that 
  \begin{itemize}
    \item $\alpha$ enters $r$ and $s$ as an overpass;
    \item $\beta$ enters $r$ and $s$ as an underpass;
    \item $\alpha$ and $\beta$ may have self-crossings; however, neither $r$
      nor $s$ is in the interior of $\alpha$ or $\beta$; and
    \item $\alpha$ and $\beta$ together contain at most $c$ crossings from
      $R$ in their interiors. (If there is a crossing in the
      interior of both $\alpha$ and $\beta$, this crossing is counted only
      once.)
  \end{itemize}

\begin{lemma}
\label{l:close}
  Let $R$ be a subset of the set of crossings in $D$, let $c \in \{0,1,2,3\}$.
  Let $r$ be the crossing in $R$ which is the first of the crossings in $R$
  removed by an economical $\IImin$ move (we allow a draw). If $w(R) \leq c$,
  then $r$ and its twin $t(r)$ are $c$-close neighbors with respect to $R$.
\end{lemma}

  \begin{proof}
    Let $\alpha(r)$ and $\beta(r)$ be the preimage arcs between $r$ and $t(r)$.
    We want to verify that they satisfy the properties of the arcs from the
    definition of the close neighbors. The first two properties follow immediately
    from the definition of preimage arcs.

    Next, we want to check that neither $r$ nor $t(r)$ is the interior of
    $\alpha(r)$ or $\beta(r)$. For contradiction, let us assume that this not
    the case. For example, suppose that $r$ also lies  in the interior of  $\alpha(r)$. Let $\rr = (r, r^1, \dots, r^b)$ be the initial
    crossing corresponding to $r$. The $(b+1)$st Reidemeister move in $\DD$
    removes $\rr$. Any preceding Reidemeister move either does not affect $\rr$ at
    all, or it is a $\III$ move swapping the crossing with other crossings. In
    any case, it preserves the self-crossing of $\alpha(r)$ at
    $\rr$. However, this contradicts the fact that $\alpha(r^b)$ is an arc of a
    bigon removed by the $(b+1)$st move.

    In order to check the last property, let us assume that, for contradiction,
    $\alpha(r)$ and $\beta(r)$ together contain at least $c+1$ crossings from
    $R(x)$ in their interiors. These crossings have to be removed from the arcs
    $\alpha(r)$ and $\beta(r)$ until we reach $\alpha(r^b)$ and $\beta(r^b)$.
    They cannot be removed by an economical $\IImin$ move as $r$ is the first
    crossing from $R(x)$ removed by such a move. Thus they have to be removed
    from the arcs either by a $\Imin$ move, a wasteful $\II$ move or a $\III$ move (by swapping with
    $r$ or $t(r)$). This contradicts $w(x) \leq c$. Indeed, if only $\Imin$ and
    $\IImin$ moves are used, we get a total weight at least $c+1$ on the crossings; if at least
    one $\III$ move is used, we get a weight at least $(c+1)\cdot \frac 23$ on the
    crossings and an additional $\frac 23$ on $r$ or $t(r)$. This is in total
    more than $c$ as $c \leq 3$.
  \end{proof}

\section{The reduction}

Let $\Phi$ be a formula in conjunctive normal form satisfying the conditions
stated in Lemma~\ref{l:sat_variant} and let $n$ be the number of variables. Our
aim is to build a diagram $D(\Phi)$ by a polynomial-time algorithm such that
$\defe(D(\Phi)) \leq n$ if and only if $\Phi$ is satisfiable.

\paragraph{The variable gadget.} First we describe the variable gadget. 
For every variable $x$ we consider the diagram depicted at
Figure~\ref{f:vg} and we denote it $V(x)$.

\begin{figure}
\begin{center}
  \includegraphics[page=1]{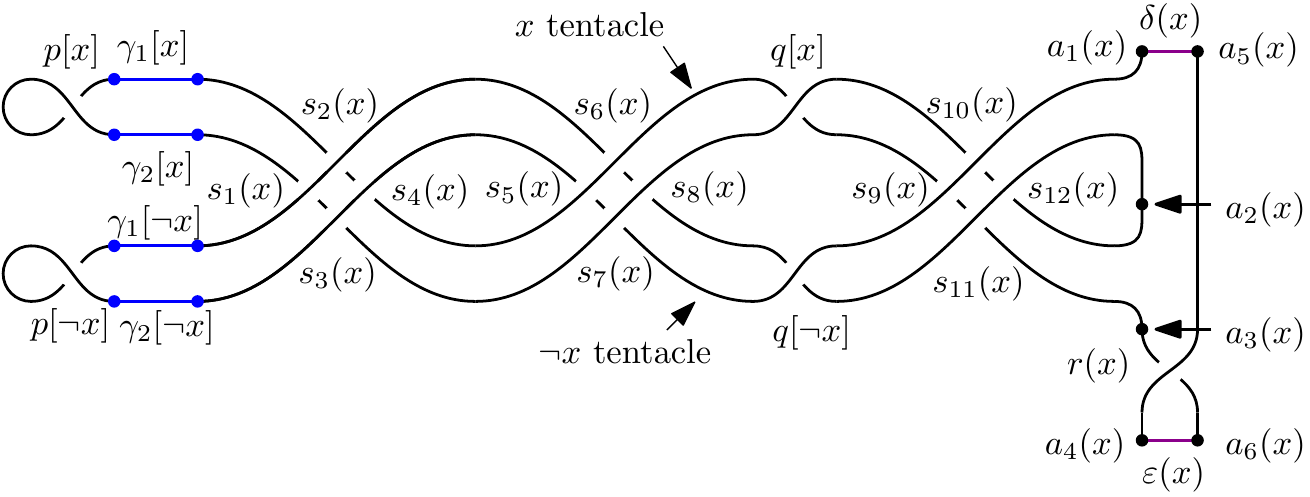}
  \caption{The variable gadget $V(x)$.} 
\label{f:vg}
\end{center}
\end{figure}

The gadget contains $17$ crossings $p[x], p[\neg x], q[x], q[\neg x]$, $r(x)$
and $s_i(x)$ for \(i \in [12]\).
The variable gadget also contains six
distinguished arcs $\gamma_i[x]$ and $\gamma_i[\neg x]$ for $i \in [2]$,
$\delta(x)$ and $\varepsilon(x)$ and six distinguished auxiliary points
$a_1(x), \dots, a_6(x)$ which will be useful later on in order to describe how
the variable gadget is used in the diagram $D(\Phi)$. 

We also call the arc between $a_1(x)$ and $a_2(x)$ which contains $\gamma_1[x]$
and $\gamma_2[x]$ the \emph{$x$ tentacle}, and similarly, the arc between $a_2(x)$ and
$a_3(x)$ which contains $\gamma_1[\neg x]$ and $\gamma_2[\neg x]$ is the
\emph{$\neg x$ tentacle}. Informally, a satisfying assignment to $\Phi$ will
correspond to the choice whether we will decide to remove first the loop at
$p[x]$ by a $\Imin$ move and simplify the $x$ tentacle or whether we remove first
the loop at $p[\neg x]$ and remove the $\neg x$ tentacle in the final
construction of $D(\Phi)$.

We also remark that in the notation, we use  square brackets for objects
that come in pairs and will correspond to a choice of literal $\ell \in \{x, \neg x\}$. This regards $p[\ell]$, $q[\ell]$, $\gamma_1[\ell]$ and $\gamma_2[\ell]$
whereas we use parentheses for the remaining objects.

\paragraph{The clause gadget.}

Given a clause $c = (\ell_1 \vee \ell_2 \vee \ell_3)$ in $\Phi$, the \emph{clause
gadget} is depicted at Figure~\ref{f:clause}.
The construction is based on the Borromean rings. It contains three
pairs of arcs (distinguished by  color) and with a slight abuse of notation,
we refer to 
each of the three pairs of arcs as a ``ring''.   Note that each ring has four pendent endpoints (or leaves) as in the picture.
Each ring corresponds to one of the literals $\ell_1$, $\ell_2$, and $\ell_3$.

\begin{figure}
\begin{center}
  \includegraphics[page=1]{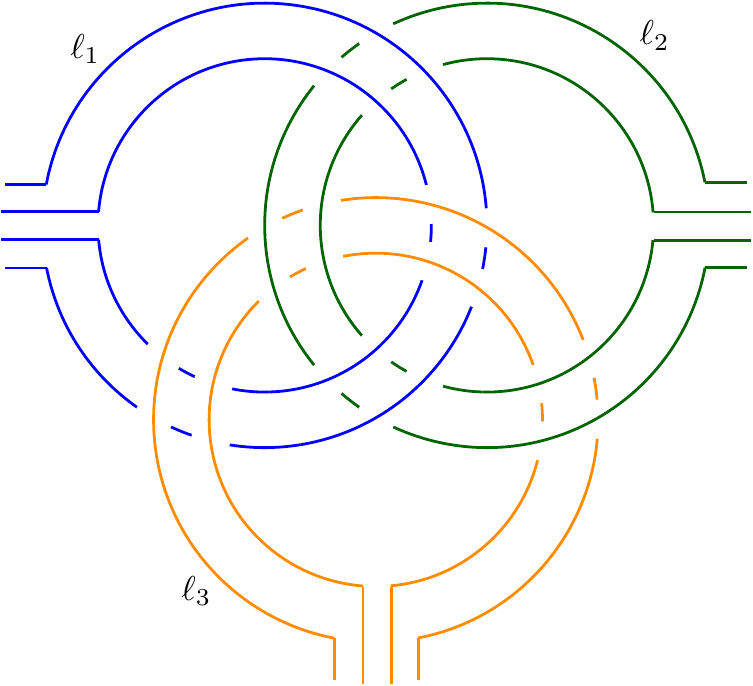}
  \caption{The clause gadget for clause $c = (\ell_1 \vee \ell_2 \vee \ell_3)$.}\label{f:clause}
\end{center}
\end{figure}

\paragraph{A blueprint for the construction.}
Now we build a blueprint for the construction of $D(\Phi)$. Let $x_1, \dots,
x_n$ be the variables of $\Phi$ and let $c_1, \dots, c_m$ be the clauses of
$\Phi$. 

For each clause $c_j = (\ell_1 \vee \ell_2 \vee \ell_3)$ we take a copy of the
graph $K_{1,3}$ (also known as the star with three leaves). We label the
vertices of degree $1$ of such a $K_{1,3}$ by the literals $\ell_1$, $\ell_2$,
and $\ell_3$. Now we draw these stars into the plane sorted along a horizontal
line; see Figure~\ref{f:blueprint}.

\begin{figure}
\begin{center}
  \includegraphics[page=1]{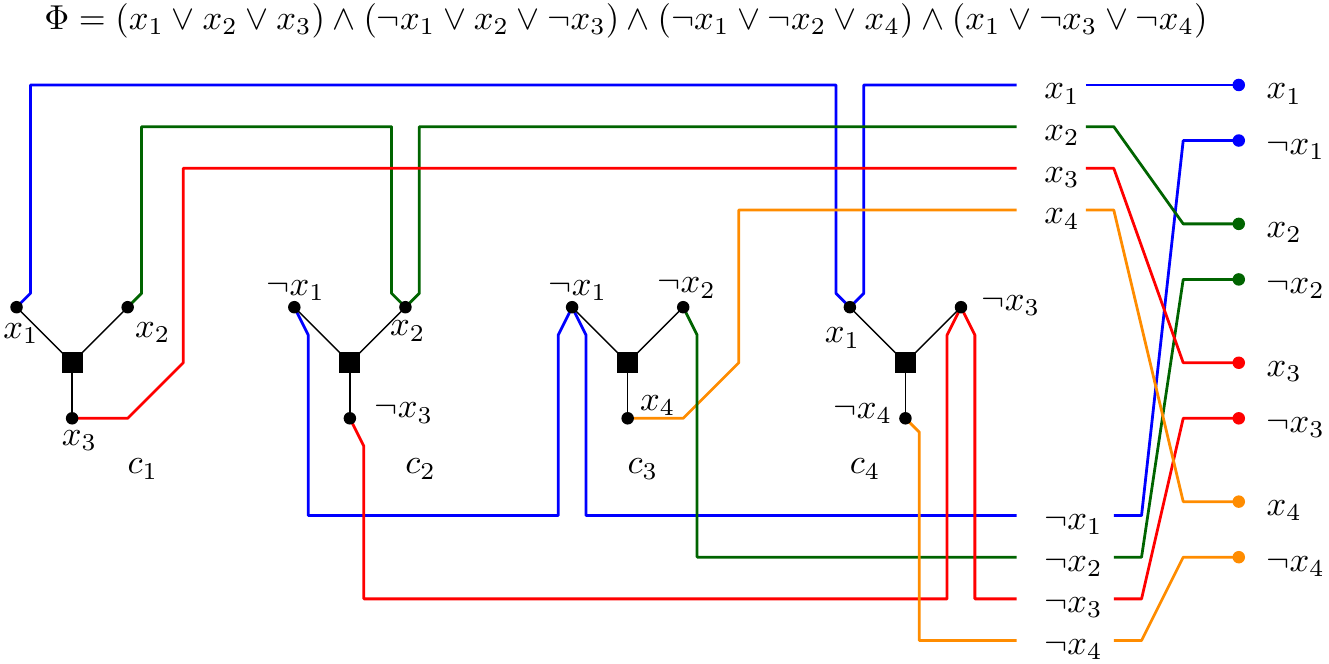}
  \caption{A blueprint for the construction of $D(\Phi)$.} 
\label{f:blueprint}
\end{center}
\end{figure}

Next for each literal $\ell \in \{x_1, \dots, x_n, \neg x_1, \dots, \neg x_n\}$
we draw a piecewise linear segment containing all vertices labelled with that literal according to
the following rules (follow Figure~\ref{f:blueprint}).

\begin{itemize}
\item
The segments start on the right of the graphs $K_{1,3}$ in the top down order $x_1,
\neg x_1$, $x_2, \neg x_2$, $\dots, x_n, \neg x_n$. 
\item
They continue to the left while
we permute them to the order $x_1, \dots, x_n$, $\neg x_1, \dots, \neg x_n$. We
also require that $x_1, \dots, x_n$ occur above the graphs $K_{1,3}$ and $\neg
x_1, \dots, \neg x_n$ occur below these graphs (everything is still on the
right of the graphs). 
\item    
Next, for each literal $\ell$ the segment for $\ell$ continues to
the left while it makes a `detour' to each vertex $v$ labelled $\ell$. 
If $v$ is not the leftmost vertex labelled $\ell$, then the detour
is performed by a `finger' of two parallel lines. We require that the finger
avoids the graphs $K_{1,3}$ except of the vertex $v$. If $v$ is the leftmost
vertex labelled $\ell$, then we perform only a half of the finger so that $v$
becomes the endpoint of the segment.
\end{itemize}

Note that the segments often intersect each other; however, for any $i \in [n]$
the segments for $x_i$ and $\neg x_i$ do not intersect (using the assumption
that no clause contains both $x_i$ and $\neg x_i$).

\paragraph{The final diagram.} Finally, we explain how to build the diagram
$D(\Phi)$ from the blueprint above.

\medskip

Step I (four parallel segments): We replace each segment for a literal $\ell$
with four parallel segments; see Figure~\ref{f:stepI}. The outer two will
correspond to the arc $\gamma_1[\ell]$ from the variable gadget and the inner two
will correspond to $\gamma_2[\ell]$; compare with Figure~\ref{f:vg}.

\begin{figure}
\begin{center}
  \includegraphics[page=2]{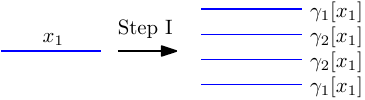}
  \caption{Step I: Replacing segments.} 
\label{f:stepI}
\end{center}
\end{figure}

\medskip

Step II (clause gadgets): We replace each copy of $K_{1,3}$ by a clause gadget for
the corresponding clause $c$; see Figure~\ref{f:stepII}. Now we aim to describe
how is the clause gadget connected to the quadruples of parallel segments
obtained in Step I. Let $v$ be a degree $1$ vertex of the $K_{1,3}$ we are just
replacing. Let $\ell$ be the literal which is the label of this vertex. Then $c$ may or may not be the leftmost clause containing a vertex labelled $\ell$. 

\begin{figure}
\begin{center}
  \includegraphics[page=3]{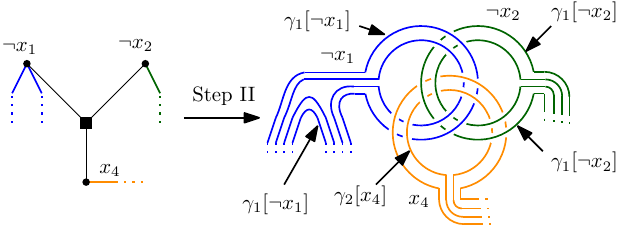}
  \caption{Step II: Replacing the $K_{1,3}$.} 
\label{f:stepII}
\end{center}
\end{figure}

If $c$ is the leftmost clause containing a vertex labelled $\ell$, then there are four parallel
segments for $\ell$ with pendent endpoints (close to the original position of
$v$) obtained in Step I. We connect them to the pendent endpoints of the clause
gadget (on the ring for $\ell$); see $\neg x_2$ and $x_4$ at
Figure~\ref{f:stepII}. Note also that at this moment the two $\gamma_1[\ell]$
arcs introduced in Step I merge as well as the two $\gamma_2[\ell]$ arcs merge.

If $c$ is not the leftmost clause labelled $\ell$ then there are four parallel
segments passing close to $v$ (forming a tip of a finger from the blueprint). We
disconnect the two segments closest to the tip of the finger and connect them to the pendent endpoints of the clause gadget (on the ring for $\ell$); see $\neg x_1$ at
Figure~\ref{f:stepII}.

\medskip

Step III (resolving crossings): If two segments in the blueprint, corresponding
to literals $\ell$ and $\ell'$ have a crossing, Step I blows up such a
crossing into $16$ crossing of corresponding quadruples. We resolve overpasses/underpasses at all these crossings in the same way. That is, one quadruple overpasses the second quadruple at all $16$ crossings; see Figure~\ref{f:stepIII}.

\begin{figure}
\begin{center}
  \includegraphics[page=4]{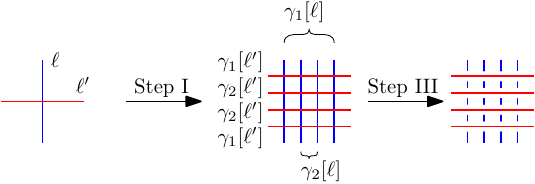}
  \caption{Step III: Resolving crossings.} 
\label{f:stepIII}
\end{center}
\end{figure}

However, we require one additional condition on the choice of
overpasses/underpasses. If $\ell$ and $\ell'$ appear simultaneously in some
clause $c$ we have $8$ crossings on the rings for $\ell$ and $\ell'$ in the
clause gadget for $c$. We can assume that the ring of $\ell$ passes over
the ring of $\ell'$ at all these crossings (otherwise we swap $\ell$ and
$\ell'$). Then for the $16$ crossings on segments for $\ell$ and $\ell'$ we
pick the other option, that is we want that the $\gamma_1[\ell]$ and
$\gamma_2[\ell]$ arcs underpass the $\gamma_1[\ell']$ and $\gamma_2[\ell']$
arcs at these crossings. This is a globally consistent choice because we assume
that there is at most one clause containing both $\ell$ and $\ell'$, this is
the third condition in the statement of Lemma~\ref{l:sat_variant}.

\medskip
STEP IV (the variable gadgets): Now, for every variable $x_i$, the segments
$\gamma_1[x_i], \gamma_2[x_i], \gamma_1[\neg x_i]$ and $\gamma_2[\neg x_i]$ do
not intersect each other. We extend them to a variable gadget as on
Figure~\ref{f:stepIV}. Namely, to the bottom left endpoints  of $\gamma_1[x_i],
\gamma_2[x_i], \gamma_1[\neg x_i]$ and $\gamma_2[\neg x_i]$ we glue the parts
of the variable gadget containing the crossings $p[x_i]$ and $p[\neg x_i]$ and
to the top left endpoints of $\gamma_1[x_i],
\gamma_2[x_i], \gamma_1[\neg x_i]$ and $\gamma_2[\neg x_i]$ we glue the
remainder of the variable gadget. At this moment, we obtain a diagram of a
link, where each link component has a diagram isotopic to the diagram on
Figure~\ref{f:vg}.

\begin{figure}
\begin{center}
  \includegraphics[page=5]{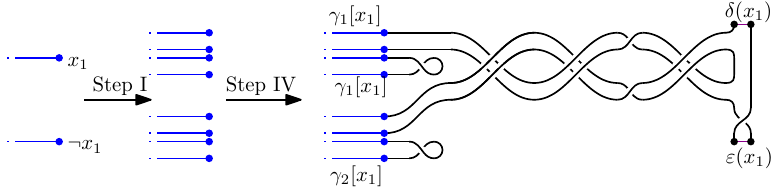}
  \caption{Step IV: Adding the variable gadgets.} 
\label{f:stepIV}
\end{center}
\end{figure}

\medskip
STEP V (interconnecting the variable gadgets): Finally, we form a connected sum of
individual components. Namely, for every $i \in [n-1]$ we perform the knot sum
along the arcs $\delta(x_i)$ and $\varepsilon(x_{i+1})$ by removing them and
identifying $a_4(x_i)$ with $a_1(x_{i+1})$ and $a_6(x_i)$ with $a_5(x_{i+1})$ as on Figure~\ref{f:stepV}. The arcs
$\delta(x_1)$ and $\varepsilon(x_n)$ remain untouched.
This way we obtain the desired diagram $D(\Phi)$; see
Figure~\ref{f:DPhi}.

\begin{figure}
\begin{center}
  \includegraphics[page=6]{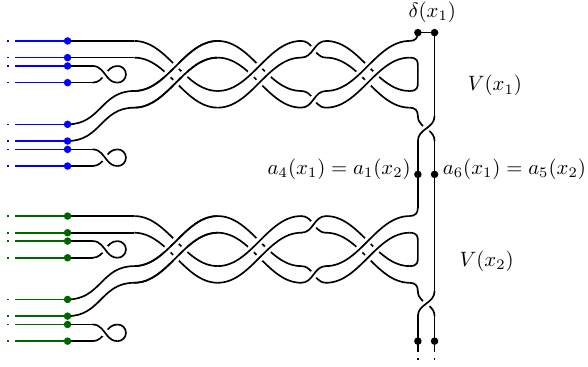}
  \caption{Step V: Interconnecting the variable gadgets.} 
\label{f:stepV}
\end{center}
\end{figure}

\begin{figure}
\begin{center}
  \includegraphics[page=7]{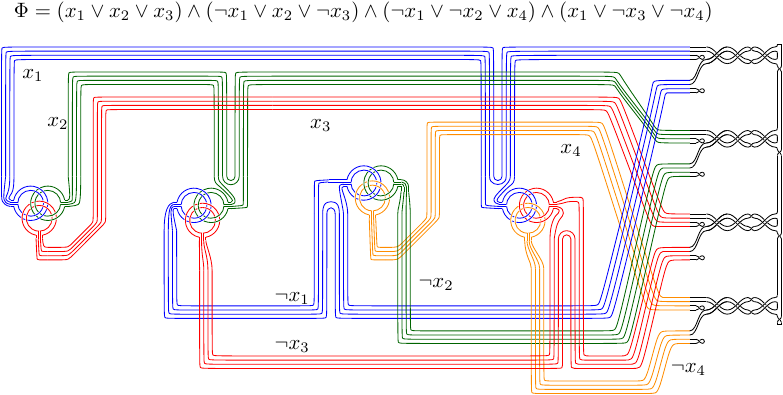}
  \caption{The final construction for the formula $\Phi = (x_1 \vee x_2 \vee x_3)
  \wedge (\neg x_1 \vee x_2 \vee \neg x_3) \wedge (\neg x_1 \vee \neg x_2 \vee x_4)
\wedge (x_1 \vee \neg x_3 \vee \neg x_4)$. For simplicity of the picture, we do
not visualize how the crossings are resolved in Step III. (Unfortunately, we 
cannot avoid tiny pictures of gadgets.)} 
\label{f:DPhi}
\end{center}
\end{figure}

The core of the \NP-hardness reduction is the following theorem.
\begin{theorem}
\label{t:sati_defe}
Let $\Phi$ be a formula in conjunctive normal form satisfying the conditions in
  the statement of Lemma~\ref{l:sat_variant}.
Then $\defe(D(\Phi)) \leq n$ if and only if $\Phi$ is satisfiable.
\end{theorem}

Theorem~\ref{t:main} immediately follows from Theorem~\ref{t:sati_defe} and
Lemma~\ref{l:sat_variant}: 

\begin{proof}[Proof of Theorem~\ref{t:main} modulo Theorem~\ref{t:sati_defe}]
Due to the definition of the defect, the minimum number of Reidemeister moves
required to untangle $D$ equals $\frac12(\defe(D(\Phi)) - \cross(D(\Phi))$. Therefore,
setting $k = \frac12(n + \cross(D(\Phi))$, Theorem~\ref{t:sati_defe} gives that
$D(\Phi)$ can be untangled with at most $k$ moves if and only if $\Phi$ is
satisfiable. This gives the required \NP-hardness via Lemma~\ref{l:sat_variant}.
(Note also, that $D(\Phi)$ and $k$ can be constructed in polynomial time in
the size of $\Phi$.)
\end{proof}

The remainder of this section is devoted to the proof of Theorem~\ref{t:sati_defe}.

\subsection{Satisfiable implies small defect}

The purpose of this subsection is to prove the `if' implication of
Theorem~\ref{t:sati_defe}. That is we are given a satisfiable $\Phi$ and we
aim to show that $\defe(D(\Phi)) \leq n$.

Let us consider a satisfying assignment. For any literal $\ell$ assigned
$\TRUE$
we first remove the loop at the $p[\ell]$ vertex in the variable gadget (see
Figure~\ref{f:vg}) by a $\Imin$ move. This way, we use one $\Imin$ move on each
variable gadget, that is $n$ such moves. Next we aim to show that it is
possible to finish the untangling of the diagram by $\IImin$ moves only. As soon
as we do this, we get an untangling with defect $n$ by
Lemma~\ref{l:defect_exact} which will finish the proof.

Thus it remains to finish the untangling with $\IImin$ moves only. We again pick
$\ell$ assigned $\TRUE$ and we start shrinking the $\ell$ tentacle by $\IImin$
moves. This way we completely shrink $\gamma_1[\ell]$ and $\gamma_2[\ell]$ as
due to the construction as all arcs that meet $\gamma_1[\ell]$ simultaneously
meet $\gamma_2[\ell]$ and vice versa. See Figure~\ref{f:tent_red} for the
initial $\Imin$ move and a few initial $\IImin$ moves. 
Furthermore we can continue shrinking the $\ell$ tentacle until we get a loop
next to the $q[\ell]$ vertex; see Figure~\ref{f:positive_shrunk}.

\begin{figure}
\begin{center}
  \includegraphics{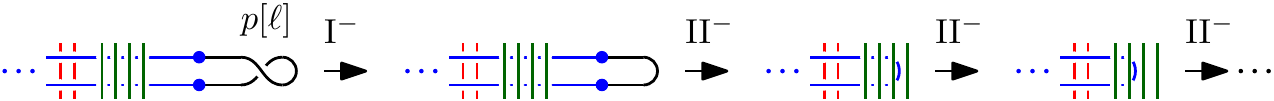}
  \caption{Initial simplifications.} 
\label{f:tent_red}
\end{center}
\end{figure}

\begin{figure}
\begin{center}
  \includegraphics[page=8]{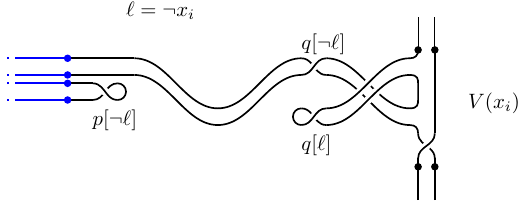}
  \caption{The $\ell$ tentacle was shrunk to a loop next to $q[\ell]$. In this
  example we have $\ell = \neg x_i$.} 
\label{f:positive_shrunk}
\end{center}
\end{figure}

We continue the same process for every literal $\ell$ assigned $\TRUE$. In the
intermediate steps, some of the other arcs meeting $\gamma_1[\ell]$ and
$\gamma_2[\ell]$ might have already been removed. However, it is still possible
to simplify the $\ell$ tentacle as before.
  See Figure~\ref{f:positive_shrunk_global} for the result after
  shrinking all tentacles assigned $\TRUE$.

\begin{figure}
\begin{center}
  \includegraphics[page=9]{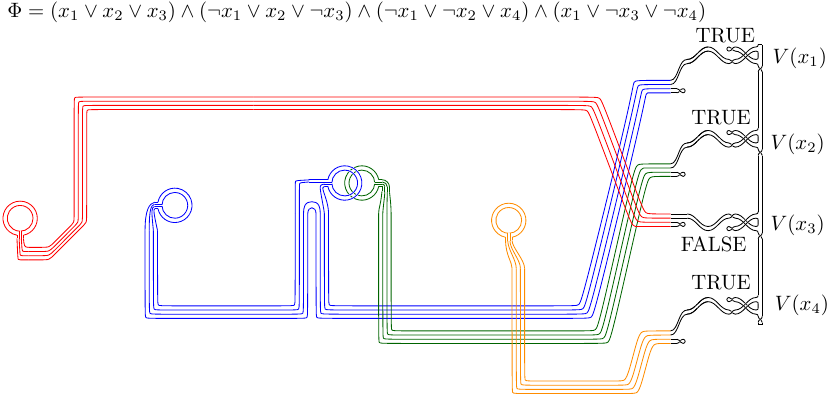}
  \caption{Simplifying the tentacles according to a satisfying assignment $x_1=
  x_2 = x_4 = \TRUE$, $x_3 = \FALSE$.} 
\label{f:positive_shrunk_global}
\end{center}
\end{figure}

Because we assume that we started with a satisfying assignment, in each clause
gadget at least one ring among the three Borromean rings disappears.
Consequently, if there are two remaining rings in some clause gadget, then they
can be pulled apart from each other by $\IImin$ moves as on
Figure~\ref{f:borromean_simplify}.

\begin{figure}
\begin{center}
  \includegraphics[page=10]{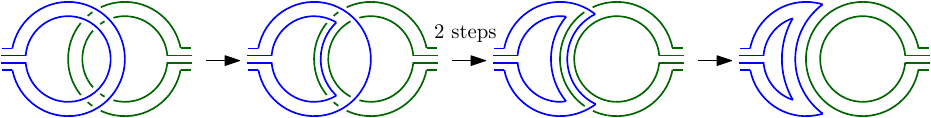}
  \caption{Untangling two rings in the clause gadget via $\IImin$ moves.} 
\label{f:borromean_simplify}
\end{center}
\end{figure}

After this step, for each $\ell$ assigned $\TRUE$, the $\gamma_1[\neg \ell]$ and
$\gamma_2[\neg \ell]$ form `fingers' of four parallel curves. These fingers can
be further simplified by $\IImin$ moves so that any crossings among different
fingers are removed; see Figure~\ref{f:fingers_negative}. For each variable
gadget $V(x)$ we get one of the two possible pictures at
Figure~\ref{f:two_options} left. Both of them simplify to the picture on the
right by three further $\IImin$ moves.

\begin{figure}
\begin{center}
  \includegraphics[page=11]{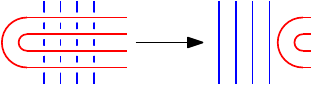}
  \caption{Simplifying $\gamma_1[\neg \ell]$ and $\gamma_2[\neg \ell]$ via
  $\IImin$ moves. First, we untangle the inner (horizontal) `finger' and then we
  untangle the outer (horizontal) `finger'.} 
\label{f:fingers_negative}
\end{center}
\end{figure}

\begin{figure}
\begin{center}
  \includegraphics[page=2]{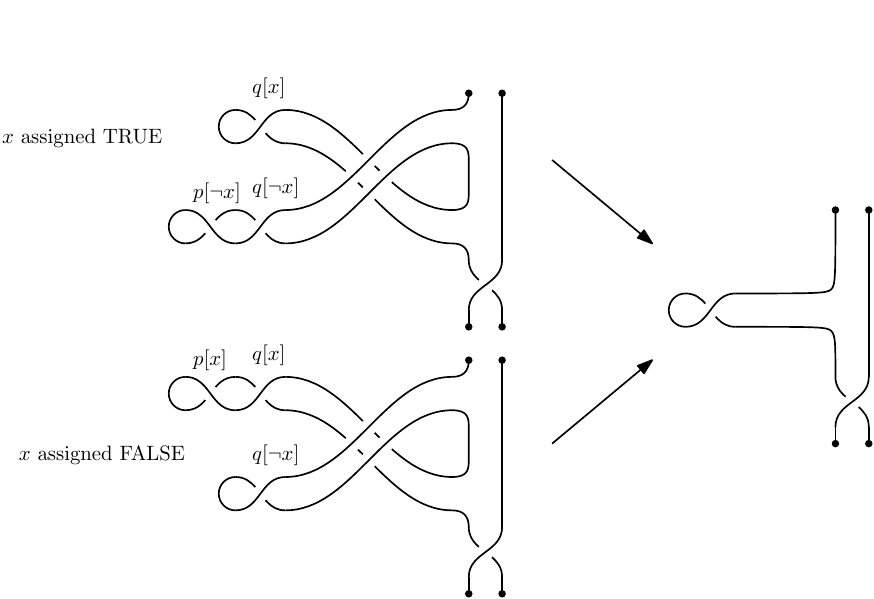}
  \caption{Results of the simplifications on the previous picture on the level
  of variable gadgets.} 
\label{f:two_options}
\end{center}
\end{figure}

Finally, we recall how are the variable gadgets interconnected (compare
the right picture at Figure~\ref{f:two_options} with Figure~\ref{f:stepV}). Then it is easy to
remove all remaining $2n$ crossings by $\IImin$ moves gradually from top to
bottom. This finishes the proof of the `if' part of Theorem~\ref{t:sati_defe}.

\subsection{Small defect implies satisfiable}

The purpose of this subsection is to prove the `only if' part of the statement
of Theorem~\ref{t:sati_defe}. Recall that this means that we assume
$\defe(D(\Phi)) \leq n$ and we want to deduce that $\Phi$ is satisfiable.
(Along  the way we will actually also deduce that $\defe(D(\Phi)) = n$.) In
this subsection, we heavily use the terminology introduced in
Section~\ref{ss:defect}.

Let $\DD = (D^0,
\dots, D^k)$ be an untangling of $D$ with $\defe(\DD) \leq n$. 
For a variable $x$ let $R(x) = \{p[x], p[\neg x], q[x], q[\neg x], s_1(x),
\dots, s_{12}(x)\}$ be the set of $16$ out of the $17$ self-crossings in the
variable gadget $V(x)$ (we leave out $r(x)$) and let the \emph{weight of $x$}, denoted by $w(x)$, be the sum
of weights of the crossings in $R(x)$. 

Our first aim is to analyze the first economical $\IImin$ move that removes
  some of the crossings in $R(x)$, using Lemma~\ref{l:close}.

\begin{claim}
\label{c:first_move}
  Let $x$ be a variable with $w(x) \leq 1$. Let $r$ be the crossing in $R(x)$ which is the first of the crossings in
  $R(x)$ removed by an economical $\IImin$ move (we allow a
  draw).
     Then one of the following cases holds
  \begin{enumerate}[(i)]
    \item $\{r, t(r)\} = \{s_1(x), s_2(x)\}$, $w(p[x]) = w(x) = 1$ and $p[x]$ is
  removed by a $\Imin$ move prior to removing $r$
      and $t(r)$.
    \item $\{r, t(r)\} = \{s_1(x), s_3(x)\}$, $w(p[\neg x]) = w(x) = 1$ and $p[\neg x]$ is
  removed by a $\Imin$ move prior to removing $r$
      and $t(r)$.
\end{enumerate}
 
\end{claim}

Before we start the proof, we remark that the condition $w(x) \leq 1$ implies
  that there are (at least $15$) crossings in $R(x)$ removed by economical
  $\IImin$ moves. In particular, $r$ in the statement always exists.

\begin{proof}
  First, we need to identify the possible pairs $\{r,t(r)\}$ where $r \in
  R(x)$. Such pairs are found by a case analysis, using Lemma~\ref{l:close}
  with $c = 1$ and $R = R(x)$.

The general strategy is the following. For each element
  of $R(x)$ we consider whether it may be $r$. We analyze possible arcs
  $\alpha$ and $\beta$ from the definition of $c$-close neighbors. 

  There are two directions in which $\alpha$ may emanate from $r$. For each direction we allow up
to one internal crossing from $R(x)$ on $\alpha$, getting a candidate
  position for $t(r)$ (even if $\alpha$ passes through other variable
  gadgets we count only the crossings from $R(x)$). We immediately disregard
  the cases when $\alpha$ passes through $r$ again (this is not allowed by the
  third item of the definition of close neighbors). We also emphasize that we
  are interested only in the cases when $\alpha$ enters the candidate $t(r)$
  as an overpass.

  Next, we refocus on $\beta$; again there are two possible directions and we
  again identify possible the possible position of $t(r)$ (this time entered
  as an underpass).

  Finally, we compare the lists of candidate positions for $t(r)$ obtained for $\alpha$ and
  for $\beta$; it must be possible to obtain $t(r)$ in both ways.

The candidate positions of $t(r)$ as an endpoint of $\alpha$ and as an
endpoint of $\beta$ are summarized in Table~\ref{t:tr}; and they can be easily
  found with the aid of Figure~\ref{f:vg}. Note that it follows from the
  construction of $D(\Phi)$ that $\delta(x)$ and $\varepsilon(x)$ are (usually)
  not in $D(\Phi)$. For considerations in the table, we denote by $\delta'(x)$
  the arc in $D(\Phi)$ between the points $a_1(x)$ and $a_5(x)$ which avoids
  $r(x)$ (equivalently, any other crossing in $V(x)$). Similarly, we let $\varepsilon'(x)$ denote
  the arc in $D(\Phi)$ between the points $a_4(x)$ and $a_6(x)$ which avoids
  $r(x)$. The table also misses values
  for $r = s_9(x)$ which requires a separate analysis.

  In order to avoid any
  ambiguity, we explain how the first row of the table is obtained,
  considering the case $r = p[x]$ (follow Figure~\ref{f:vg}):

\begin{table}
  \begin{tabular}{llll}
    Choice of $r$ & $t(r)$ in $\alpha$ & $t(r)$ in $\beta$ & Overlap \\
\hline
$p[x]$ & $\gamma_2[x]$ & $\gamma_1[x], s_2(x), s_4(x)$ &
$\emptyset$\\
$p[\neg x]$ & $\gamma_2[\neg x], s_3(x), s_4(x)$ & $\gamma_1[\neg x]$ &
$\emptyset$\\
$s_1(x)$ & $\gamma_1[\neg x], \gamma_2[\neg x], s_3(x), s_2(x)$ & $\gamma_2[x], \gamma_1[x], s_2(x), s_3(x)$ &
$s_2(x), s_3(x)$\\
$s_2(x)$ & $s_1(x), \gamma_1[\neg x]$ &
$\gamma_1[x], p[x], \gamma_2[x], s_1(x), s_4(x)$ &
$s_1(x)$\\
    $s_3(x)$ & $\gamma_2[\neg x], p[\neg x], \gamma_1[\neg x], s_1(x), s_4(x)$ &
$s_1(x), \gamma_2[x]$ & $s_1(x)$\\
    $s_4(x)$ & $s_3(x), \gamma_2[\neg x], p[\neg x]$ &
    $s_2(x), \gamma_1[x], p[x]$ & $\emptyset$\\
$s_5(x)$ & $s_6(x)$ & $s_7(x)$ & $\emptyset$\\
$s_6(x)$ & $s_5(x)$ & $s_8(x), q[\neg x]$ & $\emptyset$\\
$s_7(x)$ & $s_8(x), q[x]$ & $s_5(x)$ & $\emptyset$\\
$s_8(x)$ & $s_7(x), q[x]$ & $s_6(x), q[\neg x]$ & $\emptyset$\\
$q[x]$ & $s_7(x), s_8(x)$ & $s_9(x), s_{11}(x)$ & $\emptyset$\\
$q[\neg x]$ & $s_9(x), s_{10}(x)$ & $s_6(x), s_8(x)$ & $\emptyset$\\
    $s_9(x)$ & \multicolumn{3}{c}{separate analysis}\\
$s_{10}(x)$ & $s_9(x), q[\neg x], \delta'(x), r(x), \varepsilon'(x)$ & $s_{12}(x)$
& $\emptyset$\\
$s_{11}(x)$ & $s_{12}(x)$ & $s_9(x), q[x], r(x), \varepsilon'(x), \delta'(x)$
& $\emptyset$\\
    $s_{12}(x)$ & $s_{11}(x)$ & $s_{10}(x)$ & $\emptyset$\\

\end{tabular}
  \caption{Possible positions for $t(r)$ depending on $r$ in
  Claim~\ref{c:first_move}.}
  \label{t:tr}
\end{table}
 
 We know that $\alpha$ may emanate to the left or to the right. Emanating to the left is immediately ruled out as we reach
  $p[x]$ again in the next crossing. Emanating to the right allows $t(r)$ to be
  some crossing on $\gamma_2[x]$ or seemingly it may be $s_1(x)$ or $s_3(x)$;
  however $s_1(x)$ and $s_3(x)$ are entered by $\alpha$ as an underpass, so they
  cannot be $t(r)$. Therefore the only option, from the point of view of
  $\alpha$, is that $t(r)$ belongs to $\gamma_2[x]$, as marked in
  Table~\ref{t:tr}. Similarly when focusing on $\beta$, emanating to the left
  is immediately ruled out whereas emanating to the right allows $t(r)$ to
  belong $\gamma_1[x]$
or to be $s_2(x)$ or $s_4(x)$ (as in the table).
  We conclude that for $r = p[x]$ there is no $t(r)$ suitable both for $\alpha$
  and $\beta$, using the fact that $\gamma_1[x]$ and $\gamma_2[x]$ do not
  intersect (in the whole $D(\Phi)$). (This is marked by the $\emptyset$ in the
  overlap column.) We deduce that $r$ cannot be $p[x]$.
 
In general, for identifying the overlaps for other options of $r$, we use that
  no two of the arcs $\gamma_1[x], \gamma_2[x], \gamma_1[\neg x], \gamma_2[\neg
  x]$ intersect.

Next, we want to rule out the case $r = s_9(x)$ because this is not covered by
Table~\ref{t:tr}. Considering the possible arcs $\alpha$, we get the following
options for $t(r)$: $q[\neg x]$, $s_{10}(x)$, $\delta'(x)$, $r(x)$,
$\varepsilon'(x)$, and from the point of view of $\beta$, we have the following
options for $t(r)$: $q[x]$, $s_{11}(x)$, $r(x)$, $\varepsilon'(x)$, $\delta'(x)$.
Therefore, there are $r(x)$, $\delta'(x)$ and $\epsilon'(x)$ in the overlap (in
addition $\delta'(x)$ and $\epsilon'(x)$ may intersect). However, if we want to
reach $r(x)$, $\delta'(x)$ or $\epsilon'(x)$ with $\alpha$ we have to pass
through $s_{10}(x)$. Similarly, $\beta$ has to pass through $s_{11}(x)$. But
this violates the condition that $\alpha$ and $\beta$ together have at most
$c=1$ point of $R(x)$ in their interiors.

By checking Table~\ref{t:tr} and by the paragraph above we deduce
  that the only two options for $\{r, t(r)\}$ are $\{s_1(x), s_2(x)\}$ and
  $\{s_1(x), s_3(x)\}$. By checking possible $\alpha$ and $\beta$ we get:

\begin{itemize}
\item $\{r, t(r)\} = \{s_1(x), s_2(x)\}$, $\alpha$ is the arc directly
  connecting $s_1(x)$ and $s_2(x)$ containing no other crossings and $\beta$ is
    the arc connecting $s_1(x)$ and $s_2(x)$ passing (twice) through $p[x]$; or

\item $\{r, t(r)\} = \{s_1(x), s_3(x)\}$, $\alpha$ is
      the arc connecting $s_1(x)$ and $s_3(x)$ passing (twice) through $p[\neg
      x]$ and $\beta$ is the arc directly  connecting $s_1(x)$ and $s_3(x)$
      containing no other crossings.
\end{itemize}

  Let us first focus on the first case above.

  Before removing $s_1(x)$ and $s_2(x)$ the crossing $p[x]$ has to be removed from the
  arc $\beta$. This can be done, in principle, by a $\Imin$ move, a wasteful
  $\IImin$ move, or $p[x]$ can be 
  swapped  with $s_1(x)$ or $s_2(x)$ by a $\III$
  move before removing $s_1(x)$ and $s_2(x)$.

  Removing $p[x]$ by a $\Imin$ move is the desired conclusion $(i)$ given that
  in this case we have $w(p[x]) \geq 1$ (because of the $\Imin$ move) as well
  as $w(p[x]) \leq w(x) \leq 1$ (from the assumptions of the claim); therefore
  $w(p[x]) = w(x) = 1$
  as required.

 Removing $p[x]$ by a wasteful $\IImin$ move is impossible as we would have
 $w(p[x]) \geq 2$ whereas $w(x) \leq 1$ by the assumption of the claim.
 
 Similarly, swapping $p[x]$ with $s_1(x)$ or $s_2(x)$ is impossible as we would have $w(p[x])
 + w(s_1(x)) + w(s_2(x)) \geq 2 \cdot \frac 23$ but $w(x) \leq 1$ again.

Conclusion $(ii)$ follows analogously from the second case.
\end{proof}

Now, let us set $\ell := \ell(x) := x$ if the conclusion $(i)$ of Claim~\ref{c:first_move}
holds and $\ell := \ell(x) := \neg x$ if the conclusion $(ii)$ holds (assuming $w(x) \leq
1$). (We identify $\neg \neg x$ with $x$, that is, if $\ell = \neg x$, then
$\neg \ell = x$.) 

\begin{claim}
\label{c:twins}
  If $w(x) \leq 1$, then $p[\neg \ell]$ and $q[\neg \ell]$ are twins. In
  addition, the preimage arcs $\alpha$ and $\beta$ between $p[\neg \ell]$
  and $q[\neg \ell]$ contain $\gamma_1[\neg \ell]$ and
  $\gamma_2[\neg \ell]$.  
\end{claim}

\begin{proof}
  Let us set $R'(x) := \{p[\neg \ell], q[\ell], q[\neg \ell], s_9(x), s_{10}(x),
  s_{11}(x), s_{12}(x)\}$. Intuitively, the crossings of $R'(x)$ are those
  crossings of $R(x)$ that do not meet the arc from $q[\ell]$ to $q[\ell]$
  containing $p[\ell]$; see Figure~\ref{f:vg_choice_l}.
  
  All $r' \in R'(x)$ have $w(r') = 0$ as $w(x) =
  w(p[\ell]) = 1$ by Claim~\ref{c:first_move}. In particular, all such $r'$ are
  removed by an economical $\IImin$ move. Let $r$ be the first of these elements removed by an
economical $\IImin$ move. Let $\alpha$ and $\beta$ be the arcs from
Lemma~\ref{l:close} with $R = R'$ and $c = 0$. 

\begin{figure}
\begin{center}
  \includegraphics{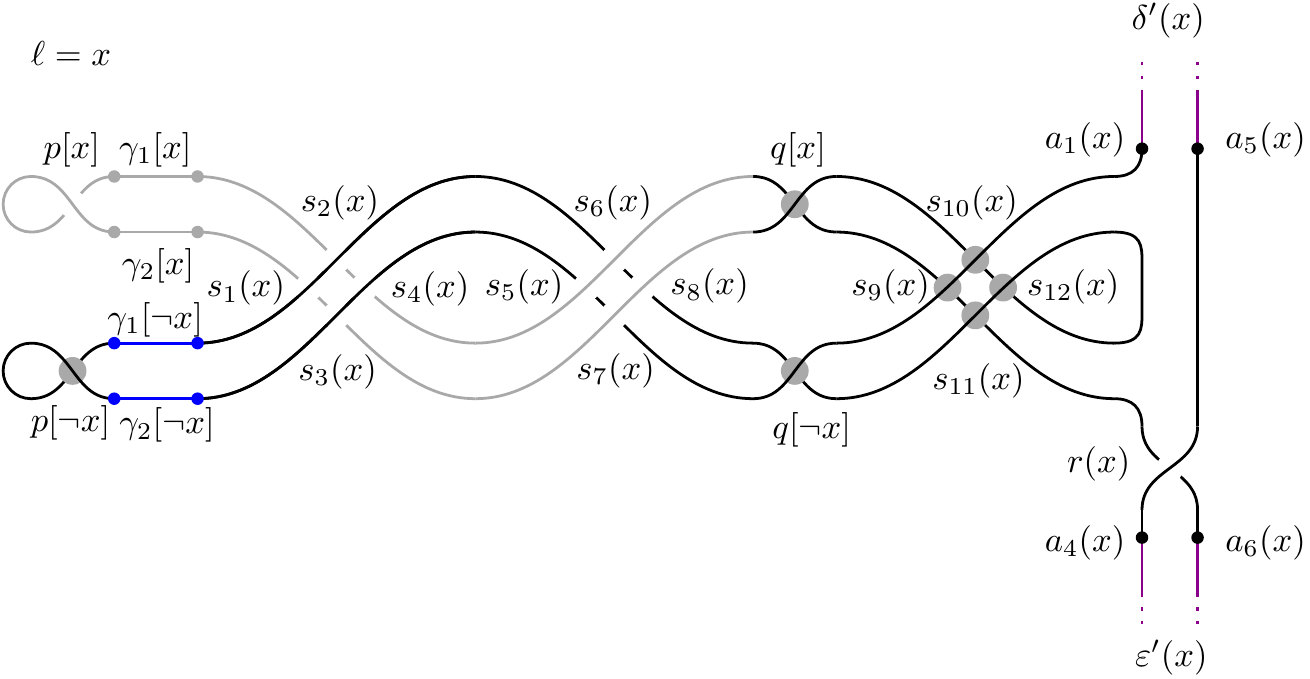}
  \caption{The crossings of $R'(x)$ in the variable gadget in case $\ell = x$.} 
\label{f:vg_choice_l}
\end{center}
\end{figure}

  Now we perform a similar inspection as in the
  proof of Claim~\ref{c:first_move}. This time $c = 0$, thus we do not allow
  any internal crossing on arcs $\alpha$ and $\beta$. Most of the cases are
  straightforward and we refer to Table~\ref{t:choice_l} for the possible
  $\alpha$ and $\beta$; $\delta'(x)$ and $\varepsilon'(x)$ play the same role
  as in the proof of the previous claim. The only exception is that we need to rule out the case
  $r = q[\ell]$ separately:

\begin{table}
\begin{tabular}{llll}
    Choice of $r$ & $t(r)$ in $\alpha$ & $t(r)$ in $\beta$ & Overlap \\
\hline
  $p[\neg \ell]$ if $\ell = x$ & $\gamma_2[\neg \ell], s_3(x), s_4(x), q[\neg
  \ell]$ &
  $\gamma_1[\neg \ell], s_6(x), s_8(x), q[\neg \ell]$ & $q[\neg \ell]$\\
  $p[\neg \ell]$ if $\ell = \neg x$ & $\gamma_2[\neg \ell], s_7(x), s_8(x),
  q[\neg \ell]$ &
  $\gamma_1[\neg \ell], s_2(x), s_4(x), q[\neg \ell]$ & $q[\neg \ell]$\\
  $q[\ell]$ &  \multicolumn{3}{c}{separate analysis}\\
  $q[\neg \ell]$ if $\ell = x$ & $s_4(x), s_3(x), \gamma_2[\neg \ell], p[\neg
  \ell], s_9(x)$ &
  $s_8(x), s_6(x), \gamma_1[\neg \ell], p[\neg \ell]$ & $p[\neg \ell]$\\
  $q[\neg \ell]$ if $\ell = \neg x$ & $s_8(x), s_7(x), \gamma_2[\neg \ell],
  p[\neg \ell]$ &
  $s_4(x), s_2(x), \gamma_1[\neg \ell], p[\neg \ell], s_9(x)$ & $p[\neg \ell]$\\
  $s_9(x)$ & $q[\neg x], s_{10}(x)$ & $q[x], s_{11}(x)$ & $\emptyset$\\
  $s_{10}(x)$ & $s_9(x), \delta(x), r(x), \epsilon(x)$ & $s_{12}(x)$ & $\emptyset$\\
  $s_{11}(x)$ & $s_{12}(x)$ & $s_9(x), r(x), \epsilon(x), \delta(x)$ & $\emptyset$\\
  $s_{12}(x)$ & $s_{11}(x)$ & $s_{10}(x)$ & $\emptyset$\\
\end{tabular}
  \caption{Possible positions for $t(r)$ depending on $r$ in Claim~\ref{c:twins}.}
  \label{t:choice_l}
\end{table}

  Let us therefore assume that $r = q[\ell]$. First we want to observe
  that both $\alpha$ and $\beta$ emanate from $q[\neg \ell]$ to the left.
  Indeed, if $\alpha$ emanates to the right, then necessarily $\ell = \neg x$
  and $t(r) = s_9(x)$ but we do not have a suitable $\beta$ for this case.
  Similarly, if $\beta$ emanates to the right, then $\ell = x$ and $t(r) =
  s_9(x)$ but we do not have a suitable $\alpha$.

  Now we know that both $\alpha$ and $\beta$ emanate to the left. This means
  that both $\alpha$ and $\beta$ are subarcs of the arc with both endpoints
  $q[\ell]$ left from $q[\ell]$ (if $\ell = x$, this is the grey arc on the
  Figure~\ref{f:vg_choice_l}). In particular, $t(r)$ has to be a selfcrossing of this arc, and
  there is only one option, namely $t(r) = p[\ell]$. However, here 
  we crucially use Claim~\ref{c:first_move}, the twin of $q[\ell]$ cannot be
  $p[\ell]$ as $p[\ell]$ is removed by a $\Imin$ move. This rules out the case
  $r= q[\ell]$.

  Therefore, it follows from Table~\ref{t:choice_l} that $\{r, t(r)\} = \{p[\neg
  \ell], q[\neg \ell]\}$. In addition, a further inspection of the variable
  gadget also reveals that $\alpha$ contains $\gamma_2[\neg \ell]$ and $\beta$
  contains $\gamma_1[\neg \ell]$ as desired.
\end{proof}

Now, we have acquired enough tools to finish the proof of the proposition.
By Claim~\ref{c:first_move}, we have $w(x) \geq 1$ for any variable $x$.
By Lemma~\ref{l:discharging}, we deduce 
$$\defe(\DD) \geq \sum_x w(x) \geq n,$$
where the sum is over all variables. On the other hand, we assume $\defe(\DD)
\leq n$. Therefore both inequalities above have to be equalities and in
particular $w(x) = 1$ for any variable $x$. In particular the assumptions of
Claims~\ref{c:first_move} and~\ref{c:twins} are satisfied for any variable $x$.

Given a variable $x$, we assign $x$ with $\TRUE$ if the conclusion $(i)$ of
Claim~\ref{c:first_move} holds (that is, if $x = \ell(x)$). Otherwise, if the conclusion $(ii)$ of
Claim~\ref{c:first_move} holds (i.e. $\neg x = \ell(x)$), we set $x$ to
$\FALSE$.
It remains to prove that we get a satisfying assignment this way.

For contradiction, suppose there is a clause $c = (\ell_1 \vee \ell_2 \vee \ell_3)$
which is not satisfied with this assignment. Let $x_i$ be the variable of
$\ell_i$, that is, $\ell_i = x_i$ or $\ell_i = \neg x_i$. The fact that $c$ is not
satisfied with the assignment above translates as $\ell(x_i) = \neg \ell_i$ for
any $i \in \{1,2,3\}$.

By Claim~\ref{c:twins}, we get that $p[\ell_i]$ and $q[\ell_i]$ are twins for
any $i \in \{1,2,3\}$. Let $R''(c)$ be the set of crossings in these
Borromeans union the sets $\{p[\ell_i], q[\ell_i]\}$ for $i \in \{1,2,3\}$.  
All the crossings in $R''(c)$ have weight $0$ and they have to be removed by
economical $\IImin$ moves as all defect is realized on points $p[\ell(x)]$ for
all variables $x$ (but $p[\ell_i] = p[\neg \ell(x_i)]$ are not among these
points). 

Let $r$ be the first removed crossing among the crossings in $R''(c)$. First,
we observe that $r$ cannot be any of $p[\ell_i]$ or $q[\ell_i]$ for $i \in
\{1,2,3\}$. This follows from Claim~\ref{c:twins} as the arcs
$\gamma_1[\ell_i]$ and $\gamma_2[\ell_i]$ contain some crossings in $R''(c)$.

Next we apply Lemma~\ref{l:close} with $R = R''(c)$ and $c = 0$.
By symmetry of the clause gadget, it is sufficent to consider the cases that
$r$ is one of the crossings $u_1, \dots, u_8$ on Figure~\ref{f:u_crossings}
between the rings for $\ell_1$ and $\ell_2$.

\begin{figure}
\begin{center}
  \includegraphics[page=2]{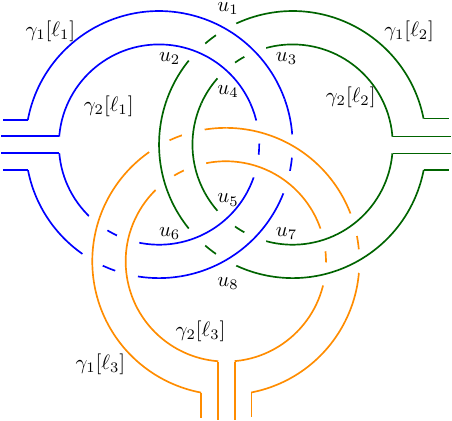}
  \caption{The clause gadget with crossings $u_1, \dots, u_8$.}
  \label{f:u_crossings}
\end{center}
\end{figure}

Let $\alpha$ and $\beta$ be the arcs between $r$ and $t(r)$ from the definition
of $c$-close neighbors. We can immediately rule out $r \in \{u_4, u_5, u_6, u_7,
u_8\}$ by an easy inspection as in Claim~\ref{c:twins} (this is easy because we
always hit a crossing from $R''(c)$ by possible $\alpha$ and $\beta$).
Therefore it remains to consider the case $r \in \{u_1, u_2, u_3\}$. 

Now let us consider the case $r = u_3$. The only option for $\alpha$ is to
emanate to the left reaching the crossing $u_1$ as emanating to the right
reaches a crossing from $R''(c)$ as an underpass. Consequently, $\beta$ has to
emanate to the right since emanating to the left would reach $u_4$. However,
before $\beta$ reaches $u_1$, it has to pass through $p[\ell_2]$ or $q[\ell_2]$
which rules out this option.

The case $r = u_2$ is ruled out analogously.

It remains to consider the case $r = u_1$. We have already ruled out the case
that the twin $t(r)$ would be $u_2$ or $u_3$ (it is sufficient to swap $r$ and
$t(r)$ in the previous considerations). Thus $\alpha$ has to emanate to the
left from $u_1$ whereas $\beta$ has to emanate to the right. The first point of
$R''(c)$ that $\alpha$ reaches is $p[\ell_1]$ or $q[\ell_1]$ while the first
point of $R''(c)$ that $\beta$ reaches is $p[\ell_2]$ or $q[\ell_2]$. Therefore
$t(r)$ must be some crossing of $\gamma_1[\ell_1]$ and $\gamma_1[\ell_2]$ while
$\alpha$ is a subarc of $\gamma_1[\ell_1]$ and $\beta$ is a subarc of
$\gamma_1[\ell_2]$. On the one hand, such crossing may exist. On the other
hand, $\alpha$ reaches such a crossing always as an underpass and $\beta$
as an overpass due to our convention in Step~III of the construction of $D(\Phi)$. Therefore, we
do not get admissible $\alpha$ and $\beta$. This contradicts the existence of
$r$. Therefore, the suggested assignment is satisfying. 
This finishes the proof.


%

\part{Hard link invariants}
\label{p:links}
\section{Intermediate Invariants}
\label{section:intermediate}

In this section we describe a family of link invariants from the
statement of Theorem~\ref{t:main2}.
The material presented here is standard and details can be found in various
textbooks; since every piecewise linear knot can be smoothed in a unique way,
we assume as we may that the knots discussed are smooth.  Throughout
Part~\ref{p:links} we work in the smooth category.  We first define:

\begin{definition}  
\label{dfn:VariousLinkInvariants}
Let \(L\) be a link in the \(3\)-sphere.  We now give a list of the invariants that we will be using; for a detailed discussion see, for example,~\cite{rolfsen}. 

\begin{enumerate}
	\item A \em smooth slice surface \em for \(L\) is an orientable surface with no closed components, properly and smoothly embedded in the \(4\)-ball, whose boundary is \(L\).  Recall that every link bounds an orientable surface in \(S^{3}\) (a Seifert surface); by pushing the interior of that Seifert surface into the \(4\)-ball we see that every link bounds a smooth slice surface.

	\item The \em \(4\)-ball Euler characteristic \em of \(L\), denoted \(\chi_{4}(L)\), is the largest integer so that \(L\) bounds a smooth slice surface of Euler characteristic \(\chi_{4}(L)\).  Since a smooth slice surface has no closed components, its Euler characteristic is at most the number of components of \(L\); in particular, \(\chi_{4}(L)\) exists.

	\item A link is called \em smoothly slice \em if it bounds a slice surface that consists entirely of disks; equivalently, the \(\chi_{4}(L)\) equals the number of components of \(L\).  Note that unlinks are smoothly slice  (but not only unlinks).
	\item The \em unlinking number\em, denoted \(u(L)\), is the smallest nonnegative integer so that \(L\) admits some diagram \(D\) so that after \(u(L)\) crossing changes on \(D\) a trivial link is obtained.
	\item The \em ribbon number\em, denoted \(u_{r}(L)\), is the smallest nonnegative integer so that \(L\) admits some diagram \(D\) so that after \(u_{r}(L)\) crossing changes on \(D\) a ribbon link is obtained (see~\cite{rolfsen} for the definition of ribbon link).
	\item The \em slicing number\em, denoted \(u_{s}(L)\), is the smallest nonnegative integer so that \(L\) admits some diagram \(D\) so that after \(u_{s}(L)\) crossing changes on \(D\) a smoothly slice link is obtained.  	
    \item Links \(L_{0},L_{1} \subset S^{3}\) are called \em concordant \em if
      there exists a smooth embedding $f: L_0\times [0,1] \rightarrow S^n \times
      [0,1]$, so that $f(L_0 \times \{0\})=L_0 \times \{0\}$ and $f(L_0 \times
      \{1\})=L_1 \times \{1\}$. Equivalently, there exists a union of smooth annuli
	\(A\), properly embedded in \(S^{3} \times [0,1]\), so that \(A \cap
	(S^{3} \times \{0\}) = L_{0} \times\{0\}\) and \(A \cap (S^{3} \times \{1\}) =
      L_{1} \times \{1\}\).
	\item The \em concordance unlinking number\em, denoted \(u_{c}(L)\), is the minimum of the unlinking number over the concordance class of \(L\).
	\item The \em concordance ribbon number\em, denoted \(u_{cr}(L)\), is the minimum of the ribbon number over the concordance class of \(L\).
	\item The \em concordance slicing number\em, denoted \(u_{cs}(L)\), is the minimum of the slicing number over the concordance class of \(L\).
	\item By transversality, every link \(L\) bounds smoothly immersed disks in \(B^{4}\) with finitely many double points.      The \em \(4\)-dimensional clasp number \em (sometimes called the \em four ball crossing number\em) of \(L\), denoted \(c_{s}(L)\), is the minimal number of double points for such disks.  
	\end{enumerate}
\end{definition}

Finally, we define:

\begin{definition}[intermediate invariant] 
\label{dfn:IntermediateInvariant}
A real valued link invariant \(i\) is called an \em intermediate invariant \em  if
\[
u \geq i \geq c_{s}
\]
\end{definition}

Many invariants are known to be intermediate (see, for example,~\cite{Shibuya}).  We list a few here:

\begin{lemma}
\label{lemma:IntermediateInvariant}
The invariants \(u\), \(u_{r}\), \(u_{s}\), \(u_{c}\), \(u_{cr}\), \(u_{cs}\) and \(c_{s}\)
are all intermediate.
\end{lemma}

\begin{proof}
It is well known that the unlink is ribbon and a ribbon link is slice, and therefore
\[
u \geq u_{r} \geq u_{s}
\text{ and  }
 u_{c} \geq u_{cr} \geq u_{cs}
 \]
Since any link is in its own concordance class we have 
\[
u_{s} \geq u_{cs}
\text{ and }
u \geq u_{c}
\]
Combining these we see that
\[
u \geq u_{r} \geq u_{s} \geq u_{cs}
\text{ and }
u \geq u_{c} \geq u_{cr} \geq u_{cs}
\]
Therefore it suffices to show that \(u_{cs} \geq c_{s}\).  To see this, decompose \(B^{4}\) (the unit ball in \(\R^{4}\)) as \(B^{4} = X_{1} \cup X_{2} \cup X_{3}\) where here:
\[
X_{1} = \big\{\vec{x} \in \R^{4} \ \big| \ 0.6 < |\vec{x}| \leq 1 \big\} \ (\cong S^{3} \times (0,1])
\]
\[
X_{2} = \big\{\vec{x} \in \R^{4} \ \big| \ 0.3 < |\vec{x}| < 0.7 \big\} \ (\cong S^{3} \times (0,1))
\]
and
\[
X_{3} = \big\{\vec{x} \in \R^{4} \ \big| \  |\vec{x}| < 0.4 \big\} \ (\cong \mathrm{int}B^{4})
\]
We used open intersections to guarantee smoothness.

Let \(L'\) be a link in the concordance class of \(L\) that minimizes \(u_{s}\), that is, \(u_{s}(L') = u_{cs}(L)\).  Since \(L\) and \(L'\) are concordant there exists disjoint annuli \(A\), smoothly embedded in \(S^{3} \times [0,1]\), so that \(A \cap (S^{3} \times\{0\}) = L \times\{0\}\) and \(A \cap (S^{3} \times\{1\}) = L' \times\{1\}\).  Since \(X_{1} \cong S^{3} \times (0,1]\), 
there exist disjoint annuli \(A_{1}\), smoothly embedded in \(X_{1}\), so that \(A_{1} \cap 
\partial X_{1} = L\) (note that \(\partial X_{1} = S^{3}\))
and 
\[
A_{1} \cap (S^{3} \times (0.6,0.7)) = L' \times (0.6,0.7)
\]

Suppose that \(L'\) has \(\mu\) components and let \(k\) denote \(u_{s}(L')\).  Let \(L''\) be a slice link obtained from \(L'\) after \(k\) crossing changes, and let \(S\) denote the disjoint union of \(\mu\) circles (so \(L'\) is an embedding of \(S\) into \(S^{3}\)).  Then there is a smooth homotopy 
\(F: S \times (0.3,0.7) \to S^{3}\) realizing  \(k\) crossing changes, that is:
	\begin{enumerate}[(i)]
	\item \(F(\cdot,t) = L'\) for \(0.6 < t < 0.7\).
	\item \(F(\cdot,t) = L''\) for \(0.3 < t < 0.4\).
	\item There are \(k\) values \(0.4 < t_{1} < \cdots < t_{k} < .06\) so that  \(F(\cdot,t_{i})\) has exactly one transverse double point.
	\item For any other value \(t\) we have that  \(F(\cdot,t)\) is a smooth embedding of \(S\).
	\end{enumerate}
Define 
\(\widehat F: S \times (0.3,0.7) \to X_{2}\) by
\[
\widehat F(p,t) = (F(p,t), t)
\]
Denote the image of \(\widehat F\) by \(A_{2}\).  Then \(A_{2}\) are \(\mu\) smoothly immersed annuli with exactly \(k\) transverse double points.  Note that 
 \[
 A_{1} \cap (X_{1} \cap X_{2})   =
 A_{2} \cap (X_{1} \cap X_{2})   
 \]

Since \(L''\) is a smoothly slice link with \(\mu\) components it bounds \(\mu\) smooth disks disjointly embedded in \(B^{4}\).  Since \(X_{3} \cong \mathrm{int}B^{4}\), this induces a smooth embedding \(D \subset X_{3}\), where here \(D\) are \(\mu\) disks, so that \(D \cap (S^{3} \times (0.3,0.4)) = L'' \times (0.3,0.4)\).  Note that
 \[
 A_{2} \cap (X_{2} \cap X_{3})   =
 D \cap (X_{2} \cap X_{3})   
 \]
It is now clear that 
\[
A_{1} \cup A_{2} \cup D
\]
is a smooth immersion of \(\mu\) disks with exactly \(k\) double points, showing that 
\[
u_{cs}(L) \leq c_{s}(L)
\]
\end{proof}

Finally we prove:

\begin{lemma}
\label{lemma:chi4geqMuCs}
Let \(L\) be a link with \(\mu\) components.  Then
\[
\chi_{4}(L) \geq \mu - 2c_{s}(L)
\] 
\end{lemma}

\begin{proof}[Sketch of proof]
This was shown by Shibuya~\cite{Shibuya}; for the convenience of the reader we sketch the proof here. Given \(\mu\) disks with \(c_{s}(L)\) double points, we endow each disk with an orientation.  One can replace a neighborhood of a double point with an annulus in a way that agrees with the chosen orientation.  The result is a smooth orientable surface \(F\) whose  Euler characteristic is \(\mu - 2c_{s}(L)\).
\end{proof}

\section{A certain signature calculation.}
\label{section:signature}

In this section we calculate the signature of certain links; this will be used in the next section.
The signature is an integer valued link invariant, defined for knots by Trotter in~\cite{trotter} and generalized for links by Murasugi in~\cite{murasugi}.

Since the signature is covered in many standard texts about knots and links we will only summarize how to calculate it:
	\begin{enumerate}
	\item  Given a link a \(L\), first construct a Seifert surface for \(L\), that is, an embedded, orientable surface \(F\), with no closed components, whose boundary is \(L\).  \(F\) need not be connected.
	\item Next construct embedded oriented curves \(\{a_{i}\}\) on \(F\) that form a basis for \(H_{1}(F;\Z)\).  
	\item Arbitrarily fix a co-orientation on each component of \(F\) (that is, directions ``above'' and ``below'' \(F\)).  We define \(a_{i}^{+}\) and \(a_{i}^{-}\) to be parallel copies of \(a_{i}\) pushed slightly above and below \(F\).
	\item The \em Seifert matrix \em \(M\) is the square matrix whose \(ij\)th entry is \(\mathrm{lk}(a_{i}^{+},a_{j}^{-})\), the linking number of \(a_{i}^{+}\) and \(a_{j}^{-}\).  Note that although the linking number is a symmetric function, the matrix \(M\) need not be symmetric.  Note also that whenever \(a_{i} \cap a_{j} = \emptyset\), the \(ij\)th entry of \(M\) is simply \(\mathrm{lk}(a_{i},a_{j})\).  
	\item The \em signature \em of \(L\), denoted \(\sigma(L)\), is the signature of the symmetric bilinear form on \(H_{1}(F;\Z)\) defined by \(M + M^{T}\); explicitly, it is the number of positive entries minus the number of  negative entries after diagonalizing \(M+M^{T}\).  
	\end{enumerate}
It is quite surprising that \(\sigma(L)\) is a link invariant, that is, independent of the choices made.  Nevertheless, it is known to be an invariant and a very useful one at that, as we shall see in the next section.  

We now define:

\begin{definition}[Whitehead Double]
\label{def:WhiteheadDouble}
Let \(L\) be a link.  A \em Whitehead double \em of \(L\) is a link obtained by taking two parallel copies of each component of \(L\) and joining them together with a clasp 
(see Figure~\ref{figure:WhiteheadDouble}). 
A Whitehead double is called \em positive \em if the crossings at the clasp are positive.  If the linking number of the two copies of each component is zero the Whitehead double is called \em untwisted\em.  It is easy to see that the untwisted positive Whitehead double is uniquely determined by \(L\).
\end{definition}

\begin{figure}
\begin{center}
\def\svgwidth{.7\textwidth}
\includesvg{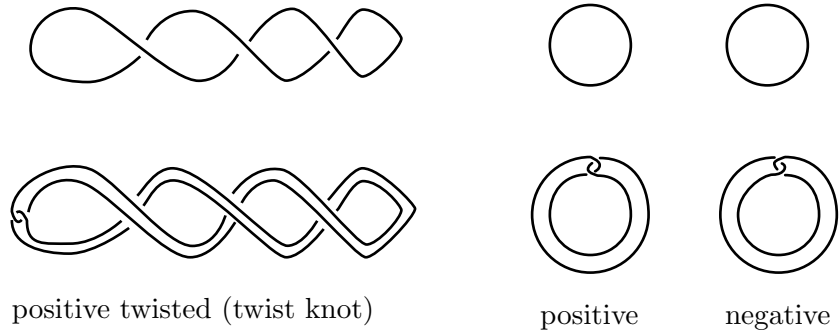}
  \caption{Some Whitehead doubles}
\label{figure:WhiteheadDouble}
\end{center}
\end{figure}

\begin{lemma}
\label{lem:SignatueOfWD1}
The signature of the positive untwisted Whitehead double of the Hopf link is \(2\).
\end{lemma}

\begin{proof}
An untwisted Whitehead double of the Hopf link bounds \(F\), two disjointly embedded once punctured tori 
(see Figure~\ref{figure:WDofHopf}). 
In that figure each torus is seen as a ``flat'' annulus with a twisted band attached near the top.  The tori are co-oriented to the positive side above the ``flat'' annulus.  In that figure we marked \(a_{1},a_{2},a_{3},a_{4}\), ordered generators for \(H_{1}(F,\Z)\).  
This gives the Seifert matrix:
\begin{figure}
\begin{center}
\def\svgwidth{\textwidth}
\includesvg{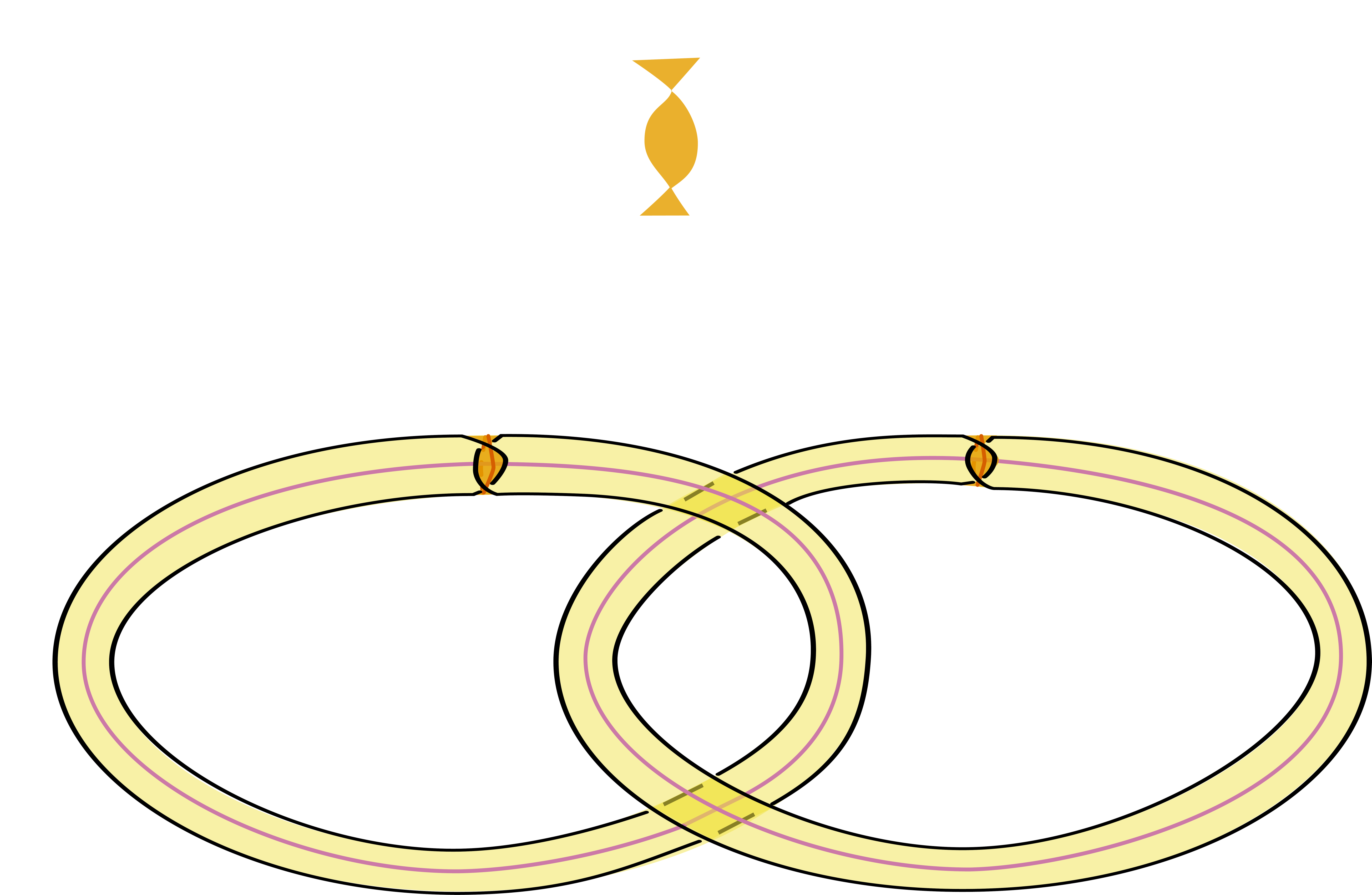}
  \caption{Top: the untwisted Whitehead double of an unknot bounds a flat annulus with a twisted band. Bottom: the positive untwisted Whitehead double of the Hopf link, a Seifert surface and generators of its homology.}
\label{figure:WDofHopf}
\end{center}
\end{figure}
\begin{equation}
\label{eq:SeifertMatrix}
M = 
\begin{pmatrix}
1 & 0 &  0 & 0 \\
1 & 0  & 0 & 1 \\
0 & 0 &   1 & 0 \\
0 & 1  & 1 & 0 
\end{pmatrix}
\end{equation}
Symmetrizing \(M\) we get
\[
M + M^{T}= 
\begin{pmatrix}
2 & 1 &  0 & 0 \\
1 & 0  & 0 & 2 \\
0 & 0 &  2 & 1 \\
0 & 2  & 1 & 0 
\end{pmatrix}
\]
A straightforward calculation shows that the signature is \(2\). 
\end{proof}

Before stating the next lemma we describe the type of link we will be dealing with.  Let \(L\) be a link with an even number of components, say \(2n\), so that \(L\) can be written as \(L = L_{1} \cup \cdots \cup L_{n}\) satisfying the following conditions:
	\begin{enumerate}
	\item \((\forall i) \ L_{i}\) is the positive untwisted Whitehead double of the Hopf link.
	\item Each \(L_{i}\) bounds \(F_{i}\), a co-oriented disjoint union of two once punctured tori, as in Figure~\ref{figure:WDofHopf}.  Let \(a_{1}^{(i)},a_{2}^{(i)},a_{3}^{(i)},a_{4}^{(i)}\) denote the generators for \(H_{1}(F_{i};\Z)\), again as in the figure.  
	\item \((\forall i \neq i') \ \ F_{i} \cap F_{i'} = \emptyset\). 
	\item \((\forall i \neq i')(\forall j,j') \ \ \mathrm{lk}(a_{j}^{(i)}, a_{j'}^{(i')}) = 0\). 
	\end{enumerate}
For a link \(L\) as above we have:

\begin{lemma}
\label{lem:SignatueOfWD2}
\(\sigma(L) = 2n\).
\end{lemma}

\begin{proof}
With the conventions above we have that \(F = F_{1} \cup \cdots \cup F_{n}\) is an oriented Seifert surface for \(L\).  We will use
\[
a_{1}^{(1)},a_{2}^{(1)},a_{3}^{(1)},a_{4}^{(1)},\dots,a_{3}^{(n)},a_{4}^{(n)}
\]
as an ordered set of generators for \(H_{1}(F;\Z)\).   We obtain a Seifert matrix that along the diagonal has \(4 \times 4\) \em  identical \em blocks, each identical to the Seifert matrix of the positive untwisted Whitehead double of the Hopf link (given in~(\ref{eq:SeifertMatrix})).  For generators corresponding to \(i \neq i'\) we have, by assumption, that \(\mathrm{lk}(a_{i}^{(j)}, a_{i'}^{(j')}) = 0\).   Thus all the remaining entries are zero.   This shows that the signature is the sum of the signatures of the \(4 \times 4\) blocks, and the lemma follows from Lemma~\ref{lem:SignatueOfWD1}.
\end{proof}

\section{Unlinking, $4$-ball Euler characteristic, and intermediate invariants.}
\label{section:4ball}

\def\lwh{\ensuremath{L_{\Phi}^{\text {\tiny WH}}}}
\def\kwh{\ensuremath{\kappa_{x_{i}}^{\text {\tiny WH}}}}
\def\kwhn{\ensuremath{\kappa_{\neg x_{i}}^{\text {\tiny WH}}}}

\def\lslice{\ensuremath{L_{\text {\tiny SLICE}}}}

In this section we will show that several link invariants are \NP-hard (see the invariants defined in Definition~\ref{dfn:VariousLinkInvariants}).
Recall the definition of the positive untwisted Whitehead double (Definition~\ref{def:WhiteheadDouble} and Figures~\ref{figure:WhiteheadDouble} and~\ref{figure:WDofHopf}).  One last piece of background we will need is a result of A. Levine~\cite{levine}. Theorem~1.1 of~\cite{levine} implies, in particular:

\begin{lemma}
\label{lem:AdamLevine}
The untwisted positive Whitehead double of the Hopf link, and that of the Borromean rings, are not smoothly slice.
\end{lemma}

We are now ready to describe our construction:

\paragraph{The construction of \lwh.}
Given a 3-SAT instance \(\Phi\), recall the link \(L_{\Phi}\) from Part~\ref{p:sublink} (Figure~\ref{figure:Lphi}), and let \(\lwh\) be its positive untwisted Whitehead double.  Note that there is a natural bijection between components before and after taking a Whitehead double; let \kwh\ denote the component corresponding to \(\kappa_{x_{i}}\) and  let \kwhn\ denote the component corresponding to \(\kappa_{\neg x_{i}}\)

  \begin{remark} If we in addition assume\footnote{This can be easily
    assumed without affecting \NP-hardness. Similarly as in the proof of
    Lemma~\ref{l:sat_variant}, we can replace $(\ell \vee \ell
    \vee \ell)$ with $(\ell \vee \neg t_1 \vee \neg t_2)$ where $t_1$ and $t_2$
    are new variables forced to be $\TRUE$ via the formula $\Psi$ from the proof
    of Lemma~\ref{l:sat_variant}.} that no clause of \(\Phi\) is of the form \((\ell
  \vee \ell \vee \ell)\) then, by
  construction, every component of \(L_{\Phi}\) is unknotted.  Since the
  Whitehead double of the unknot is unknotted, we may assume that the
  components of   \(\lwh\)  in Theorem~\ref{thm:UnlikingEtc} are all unknotted.
\end{remark}

Recall the definition of intermediate invariants (Definition~\ref{dfn:IntermediateInvariant}) and the examples given in Lemma~\ref{lemma:IntermediateInvariant}.   The goal of this section is to prove:

\begin{theorem}\label{t:equivalences}
\label{thm:UnlikingEtc}
Given a 3-SAT instance \(\Phi\) with \(n\) varaibles, let \lwh\ be the link constructed above.  Then the following are equivalent, where here \(i\) is any intermediate invariant:
	\begin{enumerate}
	\item \(\Phi\) is satisfiable.  
	\item \(u(\lwh) = n\).
	\item \(i(\lwh) = n\).
	\item \(c_{s}(\lwh) = n\).
	\item \(\chi_{4}(\lwh) = 0\).
	\item \(\lwh\) admits a smoothly slice sublink with \(n\) components.
	\end{enumerate}
\end{theorem}

Theorem~\ref{t:main2}(b)--(d) directly follows from Theorem~\ref{t:equivalences}.

\begin{proof}
 
The proof of this Theorem is split in the following steps:
\begin{enumerate}[(a)]
\item $\Phi$ is satisfiable implies that $u(\lwh) \leq n$.
\item $c_s(\lwh) \leq i(\lwh) \leq u(\lwh)$.
\item $\chi_4(\lwh) \geq 2n - 2c_s(\lwh)$.
\item If \(\chi_4(\lwh) \geq 0\) then the following two conditions hold:
	\begin{enumerate}[(d.I)]
	\item \(\chi_4(\lwh) = 0\) 
	\item $\lwh$ admits a smoothly slice sublink with $n$ components.
	\end{enumerate}
\item If $\lwh$ admits a smoothly slice sublink with $n$ components, then $\Phi$ is satisfiable.
\end{enumerate}

We first show how~(a)---(e) prove Theorem~\ref{thm:UnlikingEtc}.  Assume first that \(\Phi\) is satisfiable.  Then by~(a) and~(b) we have that \(c_{s}(\lwh) \leq n\), and by~(c) we have that \(\chi_{4}(\lwh) \geq 0\).  Then~(d.I) shows that \(\chi_{4}(\lwh) = 0\).  Working our way back, we see that \(c_{s}(\lwh) = i(\lwh) = u(\lwh) = n\), establishing $(1) \Rightarrow (2) \Rightarrow (3) \Rightarrow (4) \Rightarrow (5)$.  In addition,~(d.II) shows directly that \((5) \Rightarrow (6)\).  Finally,~(e) establishes \((6) \Rightarrow (1)\).

\bigskip\noindent
We complete the proof of Theorem~\ref{thm:UnlikingEtc} by establishing~(a)---(e):
\begin{enumerate}[(a)]
\item Suppose we have a satisfying assignment for \(\Phi\) (for this implication, {\it cf.} the proof of Theorem~\ref{thm:TirvialSublinkNPcomplete}).  
By a single crossing change we resolve the clasp of every component that correspond to a satisfied literal, that is, if \(x_{i} = \TRUE\) we change one of the crossings of the clasp of \kwh\ and if \(x_{i} = \FALSE\) we change one of the crossings of the clasp of \kwhn; as a result, the components corresponding to satisfied literals now form an unlink that is not linked with the remaining components, and we can isotope them away. 
Since the assignment is satisfying, from each copy of the Borromean rings at least one ring is removed, and the remaining components retract into the first \(n\) disks that contained the Hopf links.  In each disk we have an untwisted Whitehead double of the unknot which is itself an unknot.  Thus we see that the unlink on \(2n\) component is obtained, showing that \(u(\lwh) \leq n\).

\item By definition of intermediate invariant we have that   \(c_{s} \leq i \leq u\).

\item This is Lemma~\ref{lemma:chi4geqMuCs}.

\item Let \(F\) be a slice surface for \(\lwh\); recall that \(F\) has no closed components.
We will use the following notation:
	\begin{itemize}
	\item \(\nu\) is the number of components of \(F\); 
	\item  \(\{F_{i}\}_{i=1}^{\nu}\) are the components of \(F\); 
	\item \(g(F_{i})\) and \(\# \partial F_{i}\) denote the genus of \(F_{i}\) and the number of its boundary components.
	\end{itemize}
Murasugi~\cite[Equation~9.4 on Page~416]{murasugi} proved:
\begin{equation}
\label{equation:murasugi}
\big|\sigma(\lwh)\big| \leq 2\Big(\sum_{i=1}^{\nu}g(F_{i})\Big) + \mu - \nu  
\end{equation}
where here \(\mu\) denotes the number of components of \(\lwh\); applying Murasugi's Theorem with \(\mu = 2n\)
allows us we estimate \(\beta_{1}(F)\), the first Betti number of \(F\):

\begin{claim}
\label{c:BetaGeq2n}
\(\beta_{1}(F) \geq 2n\).
\end{claim}

\begin{proof}

Since \lwh\ has \(2n\) components we have that \(\sum_{i=1}^{\nu} \# \partial F_{i} = 2n\); this is used in \(\star\) below:
Since no component of \(F\) is closed, \(\beta_{1}(F_{i}) = 2g(F_{i}) + \#\partial F_{i} - 1\); this is used in \(\star\star\) below.
	\begin{align*}
	2n &= \big|\sigma(\lwh)\big| & \text{Lemma~\ref{lem:SignatueOfWD2}}  \nonumber \\
	&\leq  2\Big(\sum_{i=1}^{\nu}g(F_{i})\Big) + 2n - \nu & \text{Equation}~(\ref{equation:murasugi})  \nonumber \\
	&=  \sum_{i=1}^{\nu}\Big(2g(F_{i}) + \#F_{i} - 1 \Big) & \star \nonumber \\
	&= \sum_{i=1}^{\nu} \beta_{1}(F_{i}) & \star\star \nonumber \\
	&= \beta_{1}(F) \label{equation:beta1}
	\end{align*}
This completes the proof of the claim. 
\end{proof}

Next we prove:

\begin{claim}
\label{c:NuEquasN}
\(\nu \geq 2n + \chi(F)\).
\end{claim}

\begin{proof}
Let \(F_{\leq 0}\) be the \(\nu - r\) components of \(F\) that have nonpositive Euler characteristic; after reordering we may assume that the components of  \(F_{\leq 0}\) are \(F_{1},\dots,F_{\nu-r}\).  
Since the disk components of \(F\) contribute exactly \(+r\) to \(\chi(F)\) we have
\[
\sum_{i=1}^{\nu-r} \chi(F_{i}) = -r + \chi(F)
\]
Solving for \(r\) we see:
\begin{align}
\label{eq:EulerOfFi}
r &= \sum_{i=1}^{\nu-r} \Big(-\chi(F_{i})\Big) + \chi(F) \nonumber \\
   &= \sum_{i=1}^{\nu-r} \Big(2g(F_{i}) + \#F_{i} -2 \Big) + \chi(F) 
\end{align}
Disk components contribute zero to \(\beta_{1}(F)\), so \(\beta_{1}(F) = \sum_{i=1}^{\nu-r} \beta_{1}(F_{i})\).  
Thus we have (here \(\star\star\) is as in Claim~\ref{c:BetaGeq2n}):
\begin{align*}
\nu &= r + (\nu-r) \\ 
& = \sum_{i=1}^{\nu-r} \big(2g(F_{i}) + \#\partial F_{i} -2 \big) + \chi(F) + (\nu-r)  & \text{Equation}~(\ref{eq:EulerOfFi}) \\
& = \sum_{i=1}^{\nu-r} \big(2g(F_{i}) + \#\partial F_{i} - 1\big) +\chi(F)  \\
& = \sum_{i=1}^{\nu-r} \beta_{1}(F_{i})  + \chi(F) & \star\star \\
& = \beta_{1}(F) + \chi(F) \\
& \geq 2n + \chi(F) & \text{Claim}~\ref{c:BetaGeq2n} 
\end{align*}
This completes the proof of the claim. 
\end{proof}
Since \(F\) has no closed components we have that \(\nu \leq 2n\).  Thus by 
Claim~\ref{c:NuEquasN} 
we have that \(\chi(F) \leq 0\); this establishes Conclusion~(d.I). 

Now suppose that \(\chi(F) = 0\).  By 
Claim~\ref{c:NuEquasN} 
we have that \(\nu = 2n\), that is, \(F\) has exactly \(2n\) components.  Thus each component of \(F\) has exactly one boundary component.  This means that \(F_{\leq 0}\) has exactly \(2n-r\) components and each has strictly negative Euler characteristic; thus \(\chi(F_{\leq 0}) \leq -(2n-r)\).  The \(r\) disk components of \(F\) contribute exactly \(+r\) to \(\chi(F)\) we have that 
\[0 = \chi(F) \leq -(2n-r) + r = 2r -2n\]
that is,
\[
n  \leq r
\]
By construction, the link formed by the boundaries of the disks \(\{F_{\nu-r+1},\dots,F_{\nu}\}\) is an \(r\) component smoothly slice sublink of \lwh, and since \(r \geq n\), this establishes Conclusion~(d.II).

\item Finally, assume that \lwh\ admits an \(n\) component smoothly slice
  sublink \lslice.  By Lemma~\ref{lem:AdamLevine} we have that the positive
  untwisted Whitehead double of the Hopf link is not a sublink of \lslice\ and
  therefore for each \(i\) exactly one of \kwh\ and \kwhn\ is in \lslice.  If
  \kwh\ is in \lslice\ we set \(x_{i} = \FALSE\) and if \kwhn\ is in \lslice\
  we set \(x_{i} = \TRUE\). 

Using Lemma~\ref{lem:AdamLevine} again we see that \lslice\ does not admit the
positive untwisted Whitehead double of the Borromean rings as a sublink.
Therefore, from every set of Borromean rings, at least one component does \em not \em belong to \lslice; it follows that the assignment above satisfies \(\Phi\).
\end{enumerate}
\end{proof}

\section*{Acknowledgment}
We thank Amey Kaloti and Jeremy Van Horn Morris for many helpful conversations.

\bibliographystyle{alpha}
\bibliography{links_hard}

\begin{thebibliography}{dMRST18}

\bibitem[AB09]{ab-ccma-09}
Sanjeev Arora and Boaz Barak.
\newblock {\em Computational complexity: a modern approach}.
\newblock Cambridge University Press, Cambridge, 2009.

\bibitem[AHT06]{aht-cckgs-06}
Ian Agol, Joel Hass, and William Thurston.
\newblock The computational complexity of knot genus and spanning area.
\newblock {\em Transactions of the American Mathematical Society},
  358:3821--3850, 2006.

\bibitem[dMRST18]{dmrst-eR3NPh-18}
Arnaud de~Mesmay, Yo'av Rieck, Eric Sedgwick, and Martin Tancer.
\newblock Embeddability in {$\mathbb{R}^3$} is {NP}-hard.
\newblock In {\em Proceedings of the Twenty-Ninth Annual ACM-SIAM Symposium on
  Discrete Algorithms}, pages 1316--1329. Society for Industrial and Applied
  Mathematics, 2018.
\newblock Full version on arXiv:1604.00290.

\bibitem[Hak61]{h-tn-61}
Wolfgang Haken.
\newblock Theorie der {N}ormalfl{\"a}chen.
\newblock {\em Acta Mathematica}, 105(3-4):245--375, 1961.

\bibitem[HL01]{hl-nrmnu-01}
Joel Hass and Jeffrey Lagarias.
\newblock The number of {R}eidemeister moves needed for unknotting.
\newblock {\em Journal of the American Mathematical Society}, 14(2):399--428,
  2001.

\bibitem[HLP99]{hlp-ccklp-99}
Joel Hass, Jeffrey~C. Lagarias, and Nicholas Pippenger.
\newblock The computational complexity of knot and link problems.
\newblock {\em Journal of the ACM (JACM)}, 46(2):185--211, 1999.

\bibitem[HN10]{hn-udrqnr-10}
Joel Hass and Tahl Nowik.
\newblock Unknot diagrams requiring a quadratic number of {R}eidemeister moves
  to untangle.
\newblock {\em Discrete \& Computational Geometry}, 44(1):91--95, 2010.

\bibitem[KL14]{kl-huct-14}
Louis~H Kauffman and Sofia Lambropoulou.
\newblock Hard unknots and collapsing tangles.
\newblock In {\em Introductory Lectures on Knot Theory}, 2014.

\bibitem[KS]{ks-cikloi-18}
Greg Kuperberg and Eric Samperton.
\newblock Coloring invariants of knots and links are often intractable.
\newblock Manuscript.

\bibitem[KT18]{kt-NPhnak-18}
Dale Koenig and Anastasiia Tsvietkova.
\newblock {NP}-hard problems naturally arising in knot theory.
\newblock arXiv:1602.08427, 2018.

\bibitem[Kup14]{Kuperberg}
Greg Kuperberg.
\newblock Knottedness is in {$\NP$}, modulo {GRH}.
\newblock {\em Adv. Math.}, 256:493--506, 2014.

\bibitem[{Lac}15]{l-pubrm-15}
Marc {Lackenby}.
\newblock {A polynomial upper bound on {R}eidemeister moves.}
\newblock {\em {Ann. Math. (2)}}, 182(2):491--564, 2015.

\bibitem[Lac16]{l-ecktn-16}
Marc Lackenby.
\newblock The efficient certification of knottedness and {T}hurston norm.
\newblock arXiv:1604.00290, 2016.

\bibitem[Lac17a]{l-ekt-17}
Marc Lackenby.
\newblock Elementary knot theory.
\newblock In {\em Lectures on Geometry (Clay Lecture Notes)}. Oxford University
  Press, 2017.

\bibitem[Lac17b]{l-schpl3m-17}
Marc Lackenby.
\newblock Some conditionally hard problems on links and 3-manifolds.
\newblock {\em Discrete \& Computational Geometry}, 58(3):580--595, 2017.

\bibitem[Lev12]{levine}
Adam~Simon Levine.
\newblock Slicing mixed bing--whitehead doubles.
\newblock {\em Journal of Topology}, 5(3):713--726, 2012.

\bibitem[Mur65]{murasugi}
Kunio Murasugi.
\newblock On a certain numerical invariant of link types.
\newblock {\em Transactions of the American Mathematical Society},
  117:387--422, 1965.

\bibitem[Pap94]{papadimitriou94}
Christos~H. Papadimitriou.
\newblock {\em Computational complexity}.
\newblock Addison-Wesley Publishing Company, Reading, MA, 1994.

\bibitem[Rol90]{rolfsen}
Dale Rolfsen.
\newblock {\em Knots and links}, volume~7 of {\em Mathematics Lecture Series}.
\newblock Publish or Perish, Inc., Houston, TX, 1990.
\newblock Corrected reprint of the 1976 original.

\bibitem[Sam18]{s-cce3mi-18}
Eric Samperton.
\newblock Computational complexity of enumerative 3-manifold invariants.
\newblock {\em arXiv preprint arXiv:1805.09275}, 2018.

\bibitem[Shi74]{Shibuya}
Tetsuo Shibuya.
\newblock Some relations among various numerical invariants for links.
\newblock {\em Osaka J. Math.}, 11:313--322, 1974.

\bibitem[Tro62]{trotter}
Hale~F Trotter.
\newblock Homology of group systems with applications to knot theory.
\newblock {\em Annals of Mathematics}, pages 464--498, 1962.

\end{thebibliography}

\end{document}